\documentclass[12pt]{article}
\usepackage[numbers,sort&compress]{natbib}
\usepackage{amssymb,amsmath,mathrsfs,amsfonts,amsthm}
\usepackage{enumerate}
\usepackage{paralist}
\usepackage{color}
\usepackage{epsfig}
\usepackage{graphicx}
\usepackage{setspace}
\usepackage{appendix}
\begin{document}
 \def\pd#1#2{\frac{\partial#1}{\partial#2}}
\def\dfrac{\displaystyle\frac}
\let\oldsection\section
\renewcommand\section{\setcounter{equation}{0}\oldsection}
\renewcommand\thesection{\arabic{section}}
\renewcommand\theequation{\thesection.\arabic{equation}}

\newtheorem{thm}{Theorem}[section]
\newtheorem{cor}[thm]{Corollary}
\newtheorem{lem}[thm]{Lemma}
\newtheorem{prop}[thm]{Proposition}
\newtheorem*{con}{Conjucture}
\newtheorem*{questionA}{Question}
\newtheorem*{thmB}{Theorem B}
\newtheorem{remark}{Remark}[section]
\newtheorem{definition}{Definition}[section]

\title{Classification of global dynamics of competition models with  nonlocal dispersals I: Symmetric kernels
\thanks{The first author is supported by Chinese NSF (No. 11501207).
The second author is  supported by  NSF  of China (No. 11431005), NSF of Shanghai  (No. 16ZR1409600).} }

\author{Xueli Bai{\thanks{E-mail: mybxl110@163.com}}\\ {Department of Applied Mathematics, Northwestern Polytechnical
University,}\\{\small 127 Youyi Road(West), Beilin 710072,
Xi'an, P. R. China.}\\ Fang Li{\thanks{Corresponding author.
E-mail: fli@cpde.ecnu.edu.cn}}\\ {Center for PDE, East China Normal
University,}\\{\small 500 Dongchuan Road, Minhang 200241,
Shanghai, P. R. China.} }

\date{}
\maketitle{}

\begin{abstract}
In this paper,  we gives a complete classification of  the global dynamics of two-species Lotka-Volterra competition models with nonlocal dispersals:
\begin{equation*}
\begin{cases}
u_t= d \mathcal{K}[u]  +u(m(x)- u- c v) &\textrm{in } \Omega\times[0,\infty),\\
v_t= D \mathcal{P}[v]  +v(M(x)-b u-  v) &\textrm{in } \Omega\times[0,\infty),\\
u(x,0)=u_0(x),~v(x,0)=v_0(x)  &\textrm{in } \Omega,
\end{cases}
\end{equation*}
where $\mathcal{K}$, $\mathcal{P}$ represent nonlocal operators, under the assumptions that the nonlocal operators  are symmetric, the models admit  two semi-trivial steady states and $0<bc\leq 1$. In particular,  when both semi-trivial steady states are locally stable, it is proved that there exist infinitely many steady states and the solution with nonnegative and nontrivial initial data converges to  some steady state in $C(\bar\Omega)\times C(\bar\Omega)$.  Furthermore, we generalize these results  to the case that competition coefficients are location-dependent and dispersal strategies are mixture of local and nonlocal dispersals.
\end{abstract}



\section{Introduction}
Dispersal is an important feature of  life histories of many organisms and thus has been a central topic in ecology.
In 1951, random diffusion was introduced to model dispersal strategies \cite{Skellam} and there are tremendous studies in this direction, see the books \cite{CC-book, OL-book}.
Though random dispersal is widely used in models from biology, it is clearly oversimplified for describing the movement of many organisms. Moreover, as a local behavior, random dispersal  essentially  describes the movements of
organisms between adjacent spatial locations. However, the possibility of a long range dispersal is well known in ecology \cite{Cain,Clark1,Clark2,schurr}, typical instances including birds fly, propagation of seeds and pollens etc.  Evoked by this, mathematicians introduce a new diffusion mode different from the random diffusion--nonlocal dispersal.  A commonly used form that integrates such long range dispersal is the following nonlocal diffusion operator \cite{Bere-C-V, Fife,Grinfeld,HMMV,Lutscher,Mogilner-E,Turchin}:
 \begin{equation*}\label{nonlocal operator}
   \mathcal{L}u :=\int_\Omega k(x,y)u(y)dy-a(x)u(x).
 \end{equation*}

It is also worth mentioning that  the nonlocal operators have been used
to model many other applied situations beyond ecology, for example in image processing \cite{Gilboa,Kindermann}, particle systems \cite{Bodnar}, coagulation models \cite{Fournier},
nonlocal anisotropic models for phase transition \cite{Alberti1,Alberti2}, mathematical finances
using optimal control theory \cite{Biswas,Jakobsen} etc. We refer the book \cite{Andreu-J-R-T} and references therein for more details.

The purpose of this paper is to understand the role played by spatial heterogeneity and nonlocal dispersals in the ecology of competing species by classifying the global dynamics of the following   model
\begin{equation}\label{original}
\begin{cases}
u_t= d \mathcal{K}[u]  +u(m(x)- u- c v) &\textrm{in } \Omega\times[0,\infty),\\
v_t= D \mathcal{P}[v]  +v(M(x)-b u-  v) &\textrm{in } \Omega\times[0,\infty),\\
u(x,0)=u_0(x),~v(x,0)=v_0(x)  &\textrm{in } \Omega,
\end{cases}
\end{equation}
where $\Omega$ is a smooth bounded domain in $\mathbb R^n$, $n\geq 1$ and $\mathcal{K}$, $\mathcal{P}$ represent nonlocal operators.
In this model, $u(x,t)$, $v(x,t)$ are the population densities of two competing species, $d, D>0$ are their dispersal rates respectively. The functions $m(x)$, $M(x)$ represent their intrinsic growth rates,
$b, c>0$ in $\bar\Omega$ are interspecific competition coefficients.

\subsection{Background and motivations}

The  model (\ref{original}) is a Lotka-Volterra type model which can be traced back to the works of Lotka and Volterra \cite{Lotka,Volterra}. Such models are widely used to describe the dynamics of biological systems in which two species interact, where predator-prey and competition are two typical situations, and play an important role in mathematical biology. To avoid being too lengthy,  we restrict our discussions to  models related to the model (\ref{original}) only.

Let us begin with the  the   simple Lotka-Volterra ODE model (which can be considered as a special case of (\ref{original}): $d=D=0$ and $M=m,u_0,v_0$ are positive constants)
\begin{equation}\label{ode}
\begin{cases}
u^\prime (t)= u(m- u- c v) &\textrm{in } [0,\infty),\\
v^\prime (t)= v(m-b u-  v) &\textrm{in } [0,\infty),\\
u(0)=u_0,~v(0)=v_0.
\end{cases}
\end{equation}
The following results about the global dynamics of (\ref{ode}) are well known:
\begin{itemize}
  \item[(i)] If $b,c< 1$ then $(\frac{1-c}{1-bc}m, \frac{1-b}{1-bc}m)$ is the global attractor;
  \item[(ii)] If $b\le1\le c$ (or $c\le 1\le b$) and $(b-1)^2+(c-1)^2\neq 0$, then $(0,m)$ (or $(m,0)$) is the global attractor;
  \item[(iii)] If $b=c=1$, for any initial data $(u_0,v_0)$, there exists $s\in [0,1]$ such that the solution of (\ref{ode}) converge to $(sm,(1-s)m)$;
  \item[(iv)] If $b,c>1$, the solution $(u,v)$ will converge to $(m,0)$/$(0,m)$/$(\frac{1-c}{1-bc}m, \frac{1-b}{1-bc}m)$ under the condition $(v_0<\frac{1-b}{1-c}u_0)$/$(v_0>\frac{1-b}{1-c}u_0)$/$(v_0=\frac{1-b}{1-c}u_0)$ respectively.
\end{itemize}

Considering the importance of dispersal strategies for species, natually, the next step is to take the diffusion of the species into consideration. If each   individual moves randomly, it leads to the following model
\begin{equation}\label{pde1}
\begin{cases}
u_t= d \Delta u   +u(m- u- c v) &\textrm{in } \Omega\times[0,\infty),\\
v_t= D \Delta v  +v(m-b u-  v) &\textrm{in } \Omega\times[0,\infty),\\
 \frac{\partial v}{\partial \gamma}= \frac{\partial v}{\partial \gamma}=0 &\textrm{on } \partial\Omega\times[0,\infty),\\
u(x,0)=u_0(x),~v(x,0)=v_0(x)  &\textrm{in } \Omega,
\end{cases}
\end{equation}
where $\gamma$ denotes the unit outer normal vector on $\partial\Omega$.
It turns out that for the first three cases, systems (\ref{ode}) and (\ref{pde1}) share lots of similarity, while the case (iv) is more delicate.
More specifically, in the cases of (i), (ii) and (iii), the globally stable equilibrium of (\ref{ode}) given  above is also globally stable as a solution of (\ref{pde1})  \cite{Ahmad,DeMottoni}. In other words, the global dynamics of the PDE model (\ref{pde1}) is independent of the initial distributions of the two species.
However,  for the case (iv),  some different and interesting phenomena happen due to the interaction between random diffusion and shape of  habitat. If $\Omega$ is convex, except for $(m,0)$ and $(0,m)$, there are no stable equilibria   \cite{Kishimoto-W}.
But, if $\Omega$ is not convex, the system (\ref{pde1}) may have a stable spatially inhomogeneous equilibrium that corresponds to the habitat segregation phenomenon  \cite{Matano-M, Mimura-E-F, Kan-on-Y}.


Later, to understand the effect of migration and spatial heterogeneity of resources,
the global dynamics of the following model
\begin{equation}\label{pde2}
\begin{cases}
u_t= d \Delta u+ u(m(x)- u- c v) &\textrm{in } \Omega\times[0,\infty),\\
v_t= D \Delta v  +v(m(x)-b u-  v) &\textrm{in } \Omega\times[0,\infty),\\
 \frac{\partial v}{ \partial\gamma}= \frac{\partial v}{ \partial\gamma}=0 &\textrm{on } \partial\Omega\times[0,\infty),\\
u(x,0)=u_0(x),~v(x,0)=v_0(x)  &\textrm{in } \Omega,
\end{cases}
\end{equation}
where $m(x)$ is nonconstant, has received extensive studies in the last two decades.  See  \cite{CC98,  HeNi, LamNi, LWW11, Lou}
and the references therein. For $0<b,c<1$, an insightful conjecture was proposed and partially verified  in \cite{Lou}:
\begin{con}
The locally stable steady state  is globally asymptotically stable.
\end{con}
Recently, this conjecture  has been completely resolved in \cite{HeNi}  provided that $0<bc\leq 1$.  Indeed,  the appearance of spatial heterogeneity greatly increase the complexity of the global dynamics of the system (\ref{pde2}). For example, when $0<b,c<1$, both coexistence and extinction phenomena happen in (\ref{pde2}) depending on the choice of competition coefficients $b,c$ and diffusion coefficients $d,D$. According to previous discussions, this is dramatically different from both the ODE system (\ref{ode}) and the PDE system (\ref{pde1}), where the distribution of resources is assumed to be constant.  Another observation is also worth mentioning. If in addition, set $d=D=0$, then (\ref{pde2}) becomes a system of two ordinary differential
equations, whose solutions converge to
$$
\left( {1-b\over 1-bc}m_+(x), {1-c\over 1-bc}m_+(x)\right)\ \ \textrm{for every}\ x\in\Omega,
$$
where $m_+(x)=\max \{ m(x), 0\}$, among all positive
continuous initial data. Thus, the introduction of migration is also crucial.
Moreover, when $bc>1$, except for very special situations mentioned in \cite{HeNi}, the global dynamics of the system (\ref{pde2}) is far from being understood. In particular, to the best of our knowledge, there is no progress for the case (iv), i.e. $b,c>1$.

Based on the importance of nonlocal dispersals, it is natural to consider
the system (\ref{pde2}) with random diffusion replaced by nonlocal versions.
Till now the studies for the corresponding nonlocal models are quite limited. See \cite{BaiLi2015, HNShen2012, LLW14} and the references therein. This paper continues the studies in \cite{BaiLi2015, LLW14}, where a type of simplified nonlocal operator is considered.

\subsection{Main results: nonlocal dispersal strategies}
In this paper,   denote $\mathbb X= C(\bar\Omega)$,  $\mathbb  X_+ =\{ u\in \mathbb X \ | \ u\geq 0\}$ and $\mathbb  X_{++}=\mathbb  X_+\setminus \{  0\}$.
For clarity, in the statements of   main results, we focus on the   nonlocal operators $\mathcal{K}$ and $\mathcal{P}$ with no flux boundary condition. To be more specific, for $\phi\in \mathbb X$, define

\noindent\textbf{(N)} $\displaystyle \mathcal{K} [\phi] = \int_{\Omega}k(x,y)\phi(y)dy- \int_{\Omega}k(y,x) dy \phi(x), \mathcal{P} [\phi] = \int_{\Omega}p(x,y)\phi(y)dy- \int_{\Omega}p(y,x) dy \phi(x),$
where the kernels $k(x, y)$, $p(x,y)$ describe the rate at which organisms move from point $y$ to point $x$.
Nonlocal operators in hostile  surroundings  or   periodic  environments will be discussed at the  last section of this paper. See \cite{HMMV} for the derivation of  different types of nonlocal operators.

Throughout this paper, unless designated otherwise, we assume that
\begin{itemize}
\item[\textbf{(C1)}]  $ m(x), M(x)\in \mathbb X$  are nonconstant.
\item[\textbf{(C2)}]  $k(x,y)$, $p(x,y)\in C(\mathbb R^n\times \mathbb R^n)$ are nonnegative and  $k(x,x), \  p(x,x)>0$ in $\mathbb R^n$. Moreover,  $\int_{\mathbb R^n} k(x,y)dy= \int_{\mathbb R^n} k(y,x)dy=1$ and $\int_{\mathbb R^n} p(x,y)dy= \int_{\mathbb R^n} p(y,x)dy=1$.
\item[\textbf{(C3)}] $k(x,y)$, $p(x,y)$   are symmetric, i.e., $k(x,y)=k(y,x)$, $p(x,y)=p(y,x)$.
\end{itemize}

To better demonstrate our main results and techniques, some explanations are in place. Let $(U(x),V(x))$  denote a nonnegative steady state of (\ref{original}),  then   there are {\it at most} three possibilities:
\begin{itemize}
\item $(U,V)= (0,0)$ is called a {\it trivial steady state};
\item $(U,V)=(u_d, 0)$ or $(U,V)=(0, v_D)$ is called a {\it semi-trivial steady state}, where $u_d$, $v_D$ are the positive solutions to single-species models
    \begin{equation}\label{singled}
    d \mathcal{K}[U]  +U(m(x)- U)=0,
    \end{equation}
    and
    \begin{equation}\label{singleD}
     D \mathcal{P}[V]  +V(M(x)- V)=0
    \end{equation}
    respectively.
\item $U>0,\ V>0$, and we call $(U,V)$ a {\it coexistence/positive steady state}.
\end{itemize}

{\it The first main result} in this paper  gives a complete classification of  the global dynamics to the competition system (\ref{original})  provided that at least one semi-trivial steady state is locally unstable.
\begin{thm}\label{thm-main}
Assume that \textbf{(C1)}-\textbf{(C3)} hold and $0<bc\leq 1$. Also assume that (\ref{original}) admits two semi-trivial steady states $(u_d, 0)$ and $(0,v_D)$. Then for the global dynamics of the system (\ref{original}) with nonlocal operators defined in \textbf{(N)}, we have the following statements:
\begin{itemize}
\item[(i)] If both $(u_d, 0)$ and $(0,v_D)$ are locally unstable, then the system (\ref{original}) admits a unique positive steady state, which is globally asymptotically stable relative to $\mathbb X_{++} \times \mathbb X_{++}$;
\item[(ii)] If $(u_d, 0)$ is locally unstable and $(0,v_D)$ is locally stable or neutrally stable, then $(0,v_D)$ is globally asymptotically stable relative to $\mathbb X_{++} \times \mathbb X_{++}$;
\item[(iii)] If $(u_d, 0)$ is locally stable or neutrally stable and $(0,v_D)$ is locally unstable, then $(u_d, 0)$ is globally asymptotically stable relative to $\mathbb X_{++} \times \mathbb X_{++}$.
\end{itemize}
\end{thm}

For competition models with {\it local dispersals}, it is known that to show global dynamics,  it suffices to demonstrate that every positive steady state is locally stable. See \cite{HsuSW1998} and references therein, where the compactness of solutions orbits is a necessary condition. This is not satisfied in the nonlocal model (\ref{original}) due to lack of regularity.

Moreover, in handling the local model (\ref{pde2}), the key  contribution  in \cite{HeNi} is the discovery of  an intrinsic relation among a positive steady state  and  a principal eigenfunction of the linearized problem at this  steady state. However,  in nonlocal models, there are difficulties determining the local stability  by linearized analysis, since principal eigenvalue might not exist.  For single-species models or semi-trivial steady states of competition models, it is known that this issue can be resolved by perturbation arguments and spectral analysis. See \cite{BZh}, \cite{HMMV} and so on. Unfortunately, as far as we are concerned, there is no progress in the studies of linearized problem at positive steady states.  Hence, we have to avoid  analyzing local stability of  positive steady state.

Fortunately,  two-species competition models with {\it nonlocal dispersals} still have the following solution structure:
\begin{itemize}
\item  if one semi-trivial steady state is locally stable while the other is locally unstable, and there is no positive steady state, then the stable one will be globally convergent;
\item if two semi-trivial steady states are both locally unstable, then there exists at least one stable positive steady state and moreover the uniqueness will  imply global convergence.
\end{itemize}
Thus, to prove Theorem \ref{thm-main}, we turn our attention back to the well-known solution structure and verify either the nonexistence  or uniqueness of positive steady state directly based on characteristics of nonlocal operators and arguments by contradiction.

{\it The second main result} concerns the   global dynamics to the competition system (\ref{original}) when both semi-trivial steady states are stable.

\begin{thm}\label{thm-main-2}
Assume that \textbf{(C1)}-\textbf{(C3)} hold and $0<bc\leq 1$. Also assume that (\ref{original}) admits two semi-trivial steady states $(u_d, 0)$ and $(0,v_D)$. For   the system (\ref{original}) with nonlocal operators defined in \textbf{(N)},   if both $(u_d, 0)$ and $(0,v_D)$ are locally stable or neutrally stable, then $bc=1$, $bu_d = v_D$ and  system (\ref{original}) has a continuum of steady states $\{(su_d, (1-s)v_D),\ 0\leq s\leq 1\}$. Moreover,   the solution of (\ref{original}) with $(u_0,v_0)\in \mathbb X_+ \times \mathbb X_+ \setminus \{ 0\}$  approaches to  a steady state in $\{(su_d, (1-s)v_D),\ 0\leq s\leq 1\}$  in  $\mathbb X\times \mathbb X$.
\end{thm}

Notice that   the solution orbits of the system (\ref{original}) are uniformly bounded, but not precompact  in $\mathbb X\times \mathbb X$ due to lack of regularity.  Thus when there are infinitely many steady states, it is highly nontrivial to demonstrate the global convergence of solutions of the system (\ref{original}) in  $\mathbb X\times \mathbb X$. Indeed, the approaches developed in the proof of Theorem \ref{thm-main-2}, which relies on energy estimates and the  repeated applications of comparison principle,  are original and quite involved. Roughly speaking, the key part of the proof consists of the following steps:
\begin{itemize}
\item Prove that there exists $T>0$ such that the solution $(u(x,t),v(x,t))$ of (\ref{original}) satisfies $u(x,t)>0$, $0<v(x,t)<v_D(x)$ or  $0<u(x,t)<u_d(x)$, $v(x,t)>0$ in $\bar\Omega$ for $t\geq T$.
\item Make use of energy estimates to prove that a subsequence of $(u(\cdot, t), v(\cdot, t))$ converges in $L^2(\Omega) \times L^2(\Omega)$ to a steady state in $\{(su_d, (1-s)v_D),\ 0\leq s\leq 1\}$.
\item  Improve the convergence of a subsequence to the convergence of $(u(\cdot, t), v(\cdot, t))$ in $L^2(\Omega) \times L^2(\Omega)$.
\item Improve the convergence of $(u(\cdot, t), v(\cdot, t))$ in $L^2(\Omega) \times L^2(\Omega)$ to that in $\mathbb X\times \mathbb X$, which is clearly optimal for the system (\ref{original}).
\end{itemize}
Our arguments thoroughly employ the structure of monotone systems and the characteristics of nonlocal operators. We strongly believe that  this approach can be generalized  to handle  monotone system without  compactness of solution orbits. We will turn to this topic in future work.

\subsection{Main results: mixed dispersal strategies}


In many species, dispersion includes both local migration and a small proportion of long-distance migration. See \cite{PM09} and the references therein.  For example,  in genetic model with partial panmixia, the diffusion term is a combination of local and nonlocal dispersals, where the nonlocal  gives the approximation for long-distance migration. See \cite{LNN16, LNS13, N11, N12-81} for modeling and related studies. Moreover, in \cite{KaoLS10, KaoLS12},  to understand the competitive advantage among different types of dispersal strategies, the authors study  the competition  system  where  the movement of one species is purely by random walk while the other species adopts a non-local dispersal strategy.

These works motivate our studies of competing species with mixed dispersal strategies as well as location-dependent competition coefficients and self-regulations.
To be more precise, we will study models with no flux boundary conditions
\begin{equation}\label{general-mix-N}
\begin{cases}
u_t= d \left\{\alpha \mathcal{K} [u]+(1-\alpha) \Delta u \right\} +u(m(x)-b_1(x)u- c(x)v) &\textrm{in } \Omega\times[0,\infty),\\
v_t= D \left\{\beta \mathcal{P} [v]+(1-\beta) \Delta v \right\} +v(M(x)-b(x)u- c_1(x)v) &\textrm{in } \Omega\times[0,\infty),\\
(1-\alpha)\partial u/\partial \gamma =(1-\beta) \partial v/\partial \gamma=0   &\textrm{on } \partial\Omega,\\
u(x,0)=u_0(x),~v(x,0)=v_0(x)  &\textrm{in } \Omega,
\end{cases}
\end{equation}
where $\mathcal{K}$, $\mathcal{P}$ are defined in \textbf{(N)}, $b_1,\ c_1$ represent self-regulations,    and $0\leq \alpha,\beta\leq 1$.
Moreover, assume that
\begin{itemize}
\item[\textbf{(C4)}]  $b(x), c(x), b_1(x), c_1(x) \in \mathbb X$.
\end{itemize}

Equipped with the techniques developed in the study of system (\ref{original}), we manage to derive
{\it the third main result} in this paper, which completely classifies the global dynamics of system (\ref{general-mix-N}) provided that
\begin{equation}\label{general-condition}
\max_{\bar\Omega} b(x) \cdot \max_{\bar\Omega} c (x)\leq \min_{\bar\Omega} b_1(x) \cdot \min_{\bar\Omega} c_1 (x).
\end{equation}

\begin{thm}\label{thm-main-mix}
Assume that \textbf{(C1)}, \textbf{(C2)}, \textbf{(C3)}, \textbf{(C4)} hold and (\ref{general-condition}) is valid. Also assume that (\ref{general-mix-N}) admits two semi-trivial steady states $(\hat{u}_d, 0)$ and $(0,\hat{v}_D)$. Then  there exist exactly four  cases:
\begin{itemize}
\item[(i)] If both $(\hat{u}_d, 0)$ and $(0,\hat{v}_D)$ are locally unstable, then the system (\ref{general-mix-N}) admits a unique positive steady state, which is globally asymptotically stable relative to $\mathbb X_{++} \times \mathbb X_{++}$;
\item[(ii)] If $(\hat{u}_d, 0)$ is locally unstable and $(0,\hat{v}_D)$ is locally stable or neutrally stable, then $(0,\hat{v}_D)$ is globally asymptotically stable relative to $\mathbb X_{++} \times \mathbb X_{++}$;
\item[(iii)] If $(\hat{u}_d, 0)$ is locally stable or neutrally stable and $(0,\hat{v}_D)$ is locally unstable, then $(\hat{u}_d, 0)$ is globally asymptotically stable relative to $\mathbb X_{++} \times \mathbb X_{++}$;
\item[(iv)]  If both $(\hat{u}_d, 0)$ and $(0,\hat{v}_D)$ are locally stable or neutrally stable, then $b(x), c(x), b_1(x), c_1(x)$ must be constants, $bc=b_1c_1$, $b\hat{u}_d = c_1\hat{v}_D$ and the system (\ref{general-mix-N}) has a continuum of steady states $\{(s\hat{u}_d, (1-s)\hat{v}_D),\ 0\leq s\leq 1\}$. Moreover,  the solution of (\ref{general-mix-N}) with
   $(u_0,v_0)\in \mathbb X_+ \times \mathbb X_+ \setminus \{ 0\}$  approaches to
    a steady state in $\{(s\hat{u}_d, (1-s)\hat{v}_D),\ 0\leq s\leq 1\}$ in
    $\mathbb X \times \mathbb X$.
\end{itemize}
\end{thm}

First of all, we point out that the assumption (\ref{general-condition}) is a straightforward generalization of the assumption $0<bc\leq 1$ in the system (\ref{original}) and does not cause any essential difficulties in the proofs.

For the proof of Theorem \ref{thm-main-mix}(i), (ii), (iii), if $\alpha,\beta\in [0,1)$, i.e., local dispersal is at least partially adopted for both species, the method  in \cite{HeNi} can be applied since solution orbits still admit compactness. But the situation is different if at least  one of $\alpha, \beta$ is equal to $1$. However, the approach developed in the proof of Theorem \ref{thm-main} can be employed to handle $\alpha,\beta\in [0,1]$ all at once.

In the proof of Theorem \ref{thm-main-mix}(iv), extra care is needed when either $\alpha=1$ or $\beta=1$.
The proof of this case mainly  follows from that of the approach in handling the case that $\alpha=\beta=1$, which has been proved in Theorem \ref{thm-main-2}. However, some modifications  are necessary due to the essential difference between local and nonlocal diffusion. We will emphasize the different parts and the corresponding adjustments in the proof.  Moreover, when $\alpha,\beta\in [0,1)$, thanks to the  compactness of solution orbits,   the convergence of solutions is known \cite{HS2006}.

At the end, we emphasize that compared with local models, {\it lack of regularity} is the key issue in the studies of models with nonlocal dispersals. The approaches and  techniques developed in this paper to overcome the difficulties caused by this issue are important contributions of our work.

This paper is organized as follows. Section 2 provides some background properties and  a general result concerning global dynamics of two-species competition models, regardless of whether the dispersal kernels are symmetric or not. Sections 3 and 4  are devoted to the proofs of Theorems \ref{thm-main} and \ref{thm-main-2} respectively. At the end,  the proof of Theorem  \ref{thm-main-mix} is included in Section 5.


\section{Preliminaries}
In this section, we prepare some background results and   describe the scheme of proofs of main results. It is worth pointing out that throughout this section, assumption \textbf{(C3)} is not imposed, i.e., the nonlocal operators can be {\it nonsymmetric}.

\subsection{Single-species model}

For the convenience of readers, we include a general result concerning single-species models with nonlocal operators. To be more specific, we consider a more general problem, which obviously covers (\ref{singled}) and (\ref{singleD}), as follows:
\begin{equation}\label{single-general}
u_t(x,t) =\mathcal{L}[u] + f(x,u) \doteq d \int_{\Omega}k(x,y)u(y,t)dy +f(x,u),
\end{equation}
where $k(x,y)$ satisfies \textbf{(C2)} and $f(x,u)$ satisfies
\begin{itemize}
\item[\textbf{(f1)}] $f\in C(\bar\Omega\times \mathbb R^+, \mathbb R)$, $f$ is $C^1$ continuous  in $u$ and $f(x,0)=0$;
\item[\textbf{(f2)}] For $u>0$, $f(x,u)/u$ is  strictly decreasing in $u$;
\item[\textbf{(f3)}] There exists $C_1>0$ such that  $d \int_{\Omega}k(x,y)dy +f(x,C_1)/C_1\leq 0$ for all $x\in\Omega$.
\end{itemize}

To study the existence of positive steady state of (\ref{single-general}), it is natural to consider the local stability of the trivial solution $u\equiv 0$, which is determined by the signs of
$$
\lambda^*=\sup \left\{\textrm{Re}\, \lambda\, |\, \lambda\in \sigma(\mathcal{L}+f_u(x,0)   \right\},
$$
where we think of  $\mathcal{L}+f_u(x,0)$ as an operator from $\mathbb X$ to $\mathbb X$.
Also, if $\lambda $ is an eigenvalue of this operator with a continuous and positive eigenfunction, we call $\lambda $ {\it principal eigenvalue}.


\begin{thm}\label{thm-single}
Under the assumptions \textbf{(C2)}, \textbf{(f1)}, \textbf{(f2)} and \textbf{(f3)}, problem (\ref{single-general}) admits a unique positive steady state in $\mathbb X$ if and only if $\lambda^*>0$. Moreover, the unique positive steady state, whenever it exists, is globally asymptotically stable relative to $ \mathbb X_{++}$, otherwise, $u\equiv 0$ is globally asymptotically stable relative to $\mathbb X_{++}$.
\end{thm}

Theorem \ref{thm-single} has been obtained in \cite{BZh} for symmetric operators in the one dimensional case and partially obtained in \cite{Coville2010}  for nonsymmetric operators of special type.
More precisely, in \cite{Coville2010},  the author only derives the existence of positive steady states in $L^{\infty}(\Omega)$ and their pointwise convergence. The idea in the proof of Theorem \ref{thm-single} originally is motivated by single-species models with local dispersal. However, if replaced by nonlocal dispersal, two additional obstacles arsie:
\begin{itemize}
\item the principal eigenvalue of the operator  $\mathcal{L}+f_u(x,0)$ might not exist;
\item the solution orbit is not precompact in $L^{\infty}(\Omega)$.
\end{itemize}
We will briefly explain how to improve the results in \cite{BZh} and \cite{Coville2010}.


\begin{proof}[Proof of Theorem \ref{thm-single}]
Since the spectrum of the operator  $\mathcal{L}+f_u(x,0)$  has been thoroughly studied in \cite{LiCovilleWang},
the arguments  in \cite[Section 6]{Coville2010} can be applied.
In particular, we just explain how to obtain the global convergence of the positive steady state
 when $\lambda^*>0$.

Similar to \cite[Section 6.1]{Coville2010}, the existence of positive steady state can proved by the construction of upper and lower solutions, denoted by $M$ and $\delta \tilde\phi$ respectively, where $M\geq C_1$, $\delta>0$ are constants and  $\tilde\phi$ is some suitable  positive function in $\mathbb X$.
Thus there exist $\hat{u}\geq \underline{u} >0$, with
$\hat{u},\ \underline{u}\in L^{\infty}(\Omega)$, such that
$$
\lim_{t\rightarrow +\infty} u(x,t; M) = \hat{u}(x)\ \textrm{and} \
\lim_{t\rightarrow +\infty} u(x,t; \delta \tilde{\phi}) = \underline{u}(x)\ \ \textrm{pointwisely}.
$$
Also $\hat{u}$ and $\underline{u}$ are positive steady states of (\ref{single-general}) in $L^{\infty}(\Omega)$. Then applying the same arguments in \cite[Page 434]{BZh}, one sees that $\hat{u},\ \underline{u}\in \mathbb X$. Thanks to Dini's Theorem, we have
\begin{equation}\label{appendix-converge}
\lim_{t\rightarrow +\infty} u(x,t; M) = \hat{u}(x)\ \textrm{and} \
\lim_{t\rightarrow +\infty} u(x,t; \delta \tilde{\phi}) = \underline{u}(x) \ \textrm{in}\ \mathbb X.
\end{equation}

Since $\hat{u},\ \underline{u}\in \mathbb X$, the arguments in \cite[Section 6.3]{Coville2010} can be applied to obtain the uniqueness of positive steady states.

Moreover, for any $u_0\in \mathbb X_+   \setminus \{ 0\}$, choose $M> \|u_0\|_{\mathbb X} +1$  large enough such $u \equiv M$ is an upper solution of (\ref{single-general}).
Thus
\begin{equation}\label{pf-single-upper}
u(x,t;u_0) \leq u(x,t;M)\leq M.
\end{equation}
Due to \textbf{(f1)} and \textbf{(f2)},
\begin{eqnarray*}
u_t = d\int_{\Omega}  k(x,y) u(y,t) dy + f(x,u)\geq  d\int_{\Omega}  k(x,y) u(y,t) dy -c_0 u,
\end{eqnarray*}
where $ c_0 = \max_{x\in\bar\Omega,\ 0<u <M} |f(x,u)|/u<\infty.$ By comparison principle, it is easy to see that $u(x,t)> 0$ in $\bar\Omega$ for $t>0$.
Notice that   $\delta$ can be arbitrarily small, hence the desired global asymptotical stability follows from uniqueness.
\end{proof}


\subsection{Competition models}

From now on, for convenience, we rewrite the  nonlocal operators defined in \textbf{(N)}  as follows
\begin{equation}\label{K-kernel}
\mathcal{K}[u] = \int_{\Omega}k(x,y)u(y)dy- a_d(x)  u(x),
\end{equation}
\begin{equation}\label{P-kernel}
\mathcal{P}[v] = \int_{\Omega}p(x,y)v(y)dy- a_D(x)   v(x),
\end{equation}
where   $a_d(x)=\int_{\Omega}k(y,x) dy,\ a_D(x)=\int_{\Omega}p(y,x) dy$.

For clarity, we will focus on competition model (\ref{original}) and always assume that
there exist two semi-trivial steady states $(u_d, 0)$ and $(0,v_D)$.


First of all,  the linearized operator of (\ref{original})  at $(u_d, 0)$ is
\begin{equation}\label{lin-d}
\mathcal{L}_{(u_d,0)} {\phi\choose\psi}={d\mathcal{K}[\phi]+[m(x)-2u_d]\phi-cu_d\psi \choose D\mathcal{P}[\psi]+[M(x)-bu_d]\psi}.
\end{equation}
Also,  the linearized operator of (\ref{original})  at $(0, v_D)$ is
\begin{equation}\label{lin-D}
\mathcal{L}_{(0, v_D)} {\phi\choose\psi}={ d \mathcal{K}[\phi] +[m (x)-cv_D]\phi  \choose D\mathcal{P}[\psi] +[M(x)-2v_D]\psi-bv_D\phi }.
\end{equation}
Denote
\begin{eqnarray}\label{PEV}
&&  \mu_{(u_d,0)}=\sup \left\{\textrm{Re}\, \lambda\, |\, \lambda\in \sigma(D\mathcal{P}+[M(x)-bu_d])   \right\} \\
&&  \nu_{(0, v_D)}= \sup \left\{\textrm{Re}\, \lambda\, |\, \lambda\in \sigma( d \mathcal{K}  +[m (x)-cv_D])   \right\}.\nonumber
\end{eqnarray}
It is known that the signs of $\mu_{(u_d,0)}$ and $\nu_{(0, v_D)}$ determine the local stability/instability of
$(u_d,0)$ and  $(0, v_D)$ respectively.  This is explicitly stated as follows and the proof is omitted since it is standard.
\begin{lem}\label{lm-signs}
Assume that the assumptions \textbf{(C1)}, \textbf{(C2)} hold. Then
\begin{itemize}
\item[(i)] $(u_d,0)$ is locally unstable if $\mu_{(u_d,0)}>0$; $(u_d,0)$ is locally stable if $\mu_{(u_d,0)}<0$; $(u_d,0)$ is neutrally stable if $\mu_{(u_d,0)}=0$.
\item[(ii)] $(0, v_D)$ is locally unstable if $\nu_{(0, v_D)}>0$; $(0, v_D)$ is locally stable if $\nu_{(0, v_D)}<0$; $(0, v_D)$ is neutrally stable if $\nu_{(0, v_D)}=0$.
\end{itemize}
\end{lem}

Remark that as explained in Section 2.1,  in general $\mu_{(u_d,0)}$ and $\nu_{(0, v_D)}$ might not be principal eigenvalues of the corresponding linearized operators.
See \cite{LiCovilleWang} and its references for more discussions.


Next, some definitions and basic properties are  included since they will be useful in the proof of main results.

\begin{definition}
Define the {\it competitive order} in $\mathbb X \times \mathbb X$:  $(u_1, v_1)\leq_c(<_c)(u_2,v_2)$ if $u_1 \leq(<) u_2$ and $ v_1\geq(>) v_2$.
\end{definition}

\begin{definition}
We say $(u,v)\in \mathbb X\times \mathbb X$ is a {\it lower(upper) solution} of the system (\ref{original}) if
\begin{equation*}
\begin{cases}
0\leq (\geq) d\mathcal{K}[u]+u(m(x)-u-cv) & \textrm{in}\; \Omega,\\
0\geq (\leq) D\mathcal{P}[v]+v(M(x)-bu-v)  & \textrm{in}\; \Omega.
\end{cases}
\end{equation*}
\end{definition}

\begin{lem}\label{lm-monotone}
Assume that $(\tilde{u}, \tilde{v})$ and $(\underline{u},\underline{v})$ are upper and lower solutions of the system (\ref{original}) respectively with $\tilde{u}, \underline{u}, \tilde{v}, \underline{v}>0$. Then
\begin{itemize}
\item[(i)]  The solution of (\ref{original}) with initial value $(\tilde{u}, \tilde{v})$ is decreasing in $t$ under the competitive order.
\item[(ii)] The solution of (\ref{original}) with initial value $(\underline{u},\underline{v})$ is increasing in $t$ under the competitive order.
\end{itemize}
\end{lem}


\begin{lem}\label{lm-upper-lower}
Assume that the assumptions \textbf{(C1)}, \textbf{(C2)} hold. Also assume that  system (\ref{original}) admits two semi-trivial steady states $(u_d, 0)$ and $(0,v_D)$.
\begin{itemize}
\item[(i)]If $\mu_{(u_d,0)}>0$, then there exists $\varepsilon_1>0$ such that  for any $0<\varepsilon\leq\varepsilon_1$ and $0<\delta\leq\varepsilon_1$, there exists an upper solution $(\tilde{u}, \tilde{v})$ of (\ref{original}) satisfying
$$
\tilde{u}=(1+\delta)u_d(x),\ 0<\tilde{v}<\varepsilon.
$$
\item[(ii)] If $\nu_{(0, v_D)}>0$, then there exists $\varepsilon_2>0$ such that  for any $0<\varepsilon\leq\varepsilon_2$ and $0<\delta\leq\varepsilon_2$, there exists a lower solution $(\underline{u}, \underline{v})$ of (\ref{original}) satisfying
$$
0<\underline{u}<\varepsilon,\ \underline{v}=(1+\delta)v_D(x).
$$
\end{itemize}
\end{lem}

The proof of Lemma \ref{lm-upper-lower} is similar to that of  \cite[Lemmas 2.3 and 2.5]{BaiLi2015} and thus the details are omitted.

The following result explains how to characterize the global dynamics of the competition model (\ref{original}) with two semi-trivial steady states.

\begin{thm}\label{thm-monotone}
Assume that the assumptions \textbf{(C1)}, \textbf{(C2)} hold. Also assume that  system (\ref{original}) admits two semi-trivial steady states $(u_d, 0)$ and $(0,v_D)$. We have the following three possibilities:
\begin{itemize}
\item[(i)]   If both $\mu_{(u_d,0)}$ and $\nu_{(0, v_D)}$,  defined in (\ref{PEV}), are positive,  the system (\ref{original}) at least has one positive steady state in $L^{\infty}(\Omega)\times L^{\infty}(\Omega)$.    If in addition, assume that the system (\ref{original}) has a unique positive steady state in $\mathbb X\times \mathbb X$, then it is globally asymptotically stable relative to
     $\mathbb X_{++} \times \mathbb X_{++}$.
\item[(ii)]   If $\mu_{(u_d,0)}$ defined in (\ref{PEV}) is positive and no positive steady states of the system (\ref{original}) exist, then the semi-trivial steady state $(0,v_D)$ is globally asymptotically stable relative to
     $\mathbb X_{++} \times \mathbb X_{++}$.
\item[(iii)]   If $\nu_{(0, v_D)}$ defined in (\ref{PEV}) is positive and the system (\ref{original}) does not admit positive steady states, then the semi-trivial steady state $(u_d,0)$ is globally asymptotically stable relative to
     $\mathbb X_{++} \times \mathbb X_{++}$.
\end{itemize}
\end{thm}

\begin{proof}
The arguments are almost the same as that of \cite[Theorem 2.1]{BaiLi2015}, where a simplified nonlocal operator is considered.
\end{proof}

It is routine to verify that Theorem \ref{thm-monotone} also holds for the system (\ref{general-mix-N}). Indeed, one sees from the proof of Theorem \ref{thm-monotone} that for models with only nonlocal dispersals, $\mu_{(u_d,0)}$ and $\nu_{(0, v_D)}$ might not be principal eigenvalues, thus the constructions of upper/lower solutions rely on the principal eigenfunctions of suitably perturbed eigenvalue problems which admit principal eigenvalues. However, when local diffusion is incorporated, the existence of principal eigenvalues is always guaranteed, which   makes the arguments standard.

It is worth pointing out that the proof of Theorem \ref{thm-monotone}(i) relies on the upper/lower solution method and  this method can only indicate the existence of positive steady state, denoted by $(u,v)$, in $L^{\infty}(\Omega)\times L^{\infty}(\Omega)$. However, according to the assumptions \textbf{(C1)}, \textbf{(C2)}, the optimal regularity should be $(u,v)\in \mathbb X\times \mathbb X $.
A natural question is when this could be true.  The following lemma provides a partial answer,
which is very important for this paper.

\begin{lem}\label{lm-continuous}
Assume that the assumptions \textbf{(C1)}, \textbf{(C2)} hold. If $bc\leq 1$, then any positive steady state of (\ref{original}) in $L^{\infty}(\Omega)\times L^{\infty}(\Omega)$ belongs to  $\mathbb X\times \mathbb X$.
\end{lem}

\begin{proof}
It follows from the proof of  \cite[Lemma 4.1]{HNShen2012}. Note that in \cite[Lemma 4.1]{HNShen2012}, it is assumed that $bc<1$. However, $bc=1$ can be handled similarly.
\end{proof}

\section{Proof of Theorem \ref{thm-main}}
To better demonstrate the proof of Theorem \ref{thm-main}, some properties of local stability and positive steady states of (\ref{original}) will be analyzed first.

The following result is about the classification of local stability.

\begin{prop}\label{prop-localstability}
Assume that \textbf{(C1)}-\textbf{(C3)} hold and $0<bc\leq 1$. Then there exist exactly four alternatives as follows.
\begin{itemize}
\item[(i)] $\mu_{(u_d,0)}>0$,  $\nu_{(0, v_D)}>0$;
\item[(ii)] $\mu_{(u_d,0)}>0$,  $\nu_{(0, v_D)}\leq 0$;
\item[(iii)] $\mu_{(u_d,0)}\leq 0$,  $\nu_{(0, v_D)}>0$;
\item[(iv)] $\mu_{(u_d,0)}= \nu_{(0, v_D)}=0$.
\end{itemize}
Moreover,  $(iv)$  holds if and only if $bc=1$ and $bu_d= v_D$.
\end{prop}

\begin{proof}
It suffice to show that when $\mu_{(u_d,0)}\leq 0$,  $\nu_{(0, v_D)}\leq 0$, that is, none of (i)-(iii) is valid, we have $\mu_{(u_d,0)}= \nu_{(0, v_D)}=0$, and furthermore $bc=1$ and $b u_d= v_D$.

Note that
$$
\mu_{(u_d,0)}= \sup_{0 \neq \psi\in L^2} \frac{\int_{\Omega} \left( D\psi \mathcal{P}[\psi]+[M(x)-bu_d]\psi^2 \right) dx }{\int_{\Omega} \psi^2 dx}\leq 0.
$$
Thus one sees that
$$
\frac{\int_{\Omega} \left( Dv_D \mathcal{P}[v_D]+[M(x)-bu_d]v_D^2 \right) dx}{\int_{\Omega} v_D^2 dx}\leq \mu_{(u_d,0)}\leq 0,
$$
and thus, due to (\ref{singleD}), it follows that
\begin{equation}\label{pf-localstability1}
\int_{\Omega} \left( v_D^3  -bu_dv_D^2 \right) dx\leq 0.
\end{equation}
Similarly, $\nu_{(0, v_D)}\leq 0$ and (\ref{singled}) give that
\begin{equation}\label{pf-localstability2}
\int_{\Omega} \left( u_d^3  -cv_Du_d^2 \right) dx\leq 0.
\end{equation}

Now by multiplying (\ref{pf-localstability2}) by $b^3$ and using the condition $0<bc\leq 1$, we have
$$
\int_{\Omega} \left( (bu_d)^3  - v_D(bu_d)^2 \right) dx\leq \int_{\Omega} \left( (bu_d)^3  - bcv_D(bu_d)^2 \right) dx =\int_{\Omega} b^3 \left( u_d^3  - cv_D u_d^2 \right) dx \leq 0,
$$
which, together with (\ref{pf-localstability1}), implies that
\begin{equation}\label{pf-localstability-key}
\int_{\Omega}(bu_d-v_D)^2(bu_d+v_D)dx\leq 0.
\end{equation}
Therefore, all previous inequalities should be equalities. Hence it is obvious that $\mu_{(u_d,0)}= \nu_{(0, v_D)}=0$, $bc=1$ and $bu_d=v_D$.

At the end, if $bc=1$ and $bu_d=v_D$, then it is easy to check that $\mu_{(u_d,0)}= \nu_{(0, v_D)}=0$.
\end{proof}

The next results indicates that whenever there exist two ordered positive steady states,  there are  infinitely many positive steady states.
Our arguments rely on exploring characteristics of nonlocal operators, as well as some integral relations inspired by \cite{HeNi}.

\begin{prop}\label{prop-steadystates}
Assume that \textbf{(C1)}, \textbf{(C2)}, \textbf{(C3)} hold and $0<bc\leq 1$. Then (\ref{original}) admits two strictly ordered continuous positive steady states $(u,v)$ and $(u^*,v^*)$ (that is w.l.o.g., $u>u^*$, $v<v^*$) if and only if $bc=1$, $b u_d= v_D$. Moreover, all the positive steady states of (\ref{original}) consist of $(su_d, (1-s)v_D)$, $0<s<1$.
\end{prop}

\begin{proof}
If $bc=1$, $b u_d= v_D$, it is routine to check that all the positive steady states of (\ref{original}) consist of $(su_d, (1-s)v_D)$, $0<s<1$, which implies (\ref{original}) admits two strictly ordered continuous positive steady states.

Now suppose that (\ref{original}) admits two different positive steady states $(u,v)$ and $(u^*,v^*)$, w.l.o.g., $u>u^*$, $v<v^*$.
We will show that  $bc=1$, $b u_d= v_D$ is valid.

First, set $w= u- u^*>0$ and $z=v-v^*<0$ and it is standard to check that
\begin{equation}\label{pf-equation-w-z}
\begin{cases}
d\mathcal{K}[w] +(m-u-cv)w-u^*w-cu^*z=0,\\
D\mathcal{P}[z] + (M-bu-v) z- bv^*w -v^* z=0.
\end{cases}
\end{equation}
Using the equation satisfied by $u$, one has
$$
d\left(u \mathcal{K}[w] - w\mathcal{K}[u]  \right) =u u^* (w+cz).
$$
This yields that
\begin{equation}\label{pf-nonpositive}
d\int_{\Omega}\left(-u \mathcal{K}[u^*] + u^* \mathcal{K}[u]  \right) \frac{w^2}{u u^* } dx =  \int_{\Omega}(w+cz)w^2 dx.
\end{equation}
We claim that $\int_{\Omega}(w+cz)w^2 dx\leq 0$.

To prove this claim, let us calculate the left hand side of (\ref{pf-nonpositive}). Note that assumption \textbf{(C3)}, i.e. $k(x,y)$ is symmetric, is important in the following computations.
\begin{eqnarray}\label{pf-exchange-xy-1}
&& d\int_{\Omega}\left(-u \mathcal{K}[u^*] + u^* \mathcal{K}[u]  \right) \frac{w^2}{u u^* } dx \cr
&=& d\int_{\Omega}\int_{\Omega} k(x,y)\left[ u^*(x)u(y)-u(x)u^*(y) \right] \frac{(u(x)-u^*(x))^2}{u(x) u^*(x) } dy dx\cr
&=&  d\int_{\Omega}\int_{\Omega} k(x,y)\left[ u^*(x)u(y)-u(x)u^*(y) \right]\left( \frac{u(x)}{u^*(x) }+\frac{u^*(x)}{u(x) } \right) dy dx,
\end{eqnarray}
where $\int_{\Omega}\int_{\Omega} k(x,y)\left[ u^*(x)u(y)-u(x)u^*(y) \right] dydx =0$ is used. By exchanging $x$ and $y$, we have
\begin{eqnarray}\label{pf-exchange-xy-2}
&& d\int_{\Omega}\left(-u \mathcal{K}[u^*] + u^* \mathcal{K}[u]  \right) \frac{w^2}{u u^* } dx \cr
&=&  d\int_{\Omega}\int_{\Omega} k(y,x)\left[ u^*(y)u(x)-u(y)u^*(x) \right]\left( \frac{u(y)}{u^*(y) }+\frac{u^*(y)}{u(y) } \right) dy dx.
\end{eqnarray}
Due to (\ref{pf-exchange-xy-1}) and (\ref{pf-exchange-xy-2}), one sees that
\begin{eqnarray*}
&& d\int_{\Omega}\left(-u \mathcal{K}[u^*] + u^* \mathcal{K}[u]  \right) \frac{w^2}{u u^* } dx \cr
&=&  {d\over 2}\int_{\Omega}\int_{\Omega} k(x,y)\left[ u^*(x)u(y)-u(x)u^*(y) \right]\left( \frac{u(x)}{u^*(x) }+\frac{u^*(x)}{u(x) } -\frac{u(y)}{u^*(y) }-\frac{u^*(y)}{u(y) } \right) dy dx\\
&=&  {d\over 2}\int_{\Omega}\int_{\Omega} k(x,y)\left[ u^*(x)u(y)-u(x)u^*(y) \right]^2\left(  \frac{1}{u(x) u(y) } -\frac{1}{u^*(x)u^*(y) }  \right) dy dx\\
&\leq & 0
\end{eqnarray*}
since $u>u^*$. The claim is proved, i.e., $\int_{\Omega}(w+cz)w^2 dx\leq 0$.

Similarly, using (\ref{pf-equation-w-z}) and the equation satisfied by $v$, we have
$$
D\left(v \mathcal{P}[z] - z\mathcal{P}[v]  \right) =v v^* (bw+z),
$$
which gives that
$$
D\int_{\Omega}\left(-v \mathcal{P}[v^*] + v^*\mathcal{P}[v]  \right) \frac{z^2}{v v^* } dx =  \int_{\Omega}(bw+z)z^2 dx.
$$
Similar to the proof of the previous claim, we obtain
\begin{eqnarray}\label{pf-key}
&& \int_{\Omega}(bw+z)z^2 dx\cr
&=&D\int_{\Omega}\left(-v \mathcal{P}[v^*] + v^*\mathcal{P}[v]  \right) \frac{z^2}{v v^* } dx\cr
&=&  {D\over 2}\int_{\Omega}\int_{\Omega} p(x,y)\left[ v^*(x)v(y)-v(x)v^*(y) \right]^2\left(  \frac{1}{v(x) v(y) } -\frac{1}{v^*(x)v^*(y) }  \right) dy dx\cr
&\geq & 0
\end{eqnarray}
since $v<v^*$.

Now we have derived two important inequalities:
\begin{equation}\label{important}
\int_{\Omega}(w+cz)w^2 dx\leq 0,\ \ \ \int_{\Omega}(bw+z)z^2 dx\geq 0.
\end{equation}
Multiplying the second one by $c^3$ and subtracting the first one, it follows that
\begin{eqnarray}\label{pf-keyrelation}
0 &\leq& \int_{\Omega}(cbw+cz)(cz)^2 dx -\int_{\Omega}(w+cz)w^2 dx\cr
&\leq &  \int_{\Omega}(w+cz)(cz)^2 dx -\int_{\Omega}(w+cz)w^2 dx\cr
&=& \int_{\Omega}(w+cz)^2(cz-w) dx,
\end{eqnarray}
where $bc\leq 1$ is used in the second inequality.
The assumption $w= u- u^*>0$ and $z=v-v^*<0$ indicates that $w+cz =0$ in $\bar\Omega$ and all the previous inequalities should be equalities. Hence we also have $bc=1$ and $bw+z =0$ (i.e., $w+cz =0$) in $\bar\Omega$.

Moreover, note that $w+cz =0$ is equivalent to $u+cv =u^* +cv^*$. Denote $R(x)=u+cv =u^* +cv^*$ for convenience.   According to the equation satisfied by $u$, $u^*$, one sees that   both  $u$ and $u^*$ are solutions of the same linear equation
$$
d \mathcal{K}[U]  + (m(x)-R(x)) U=0.
$$
Since both  $u$ and $u^*$ are positive functions in in $\mathbb X$, $u$ and $u^*$ can be regarded as the principal eigenfunctions of the nonlocal eigenvalue problem
$$
d \mathcal{K}[\phi]  +(m(x)-R(x))\phi=\lambda\phi
$$
with the principal eigenvalue being zero. It is proved in \cite{LiCovilleWang} that the principal eigenvalue is algebraically simple whenever it exists, which implies that  $u^*= \alpha u$, where $0<\alpha<1$. Similarly, it can be verified that $v^*= \beta v$, where $\beta>1$. Then using $u+cv =u^* +cv^*$ again, we have
$$
u=c{\beta-1\over 1-\alpha} v.
$$
Substitute this relation into the system satisfied by $(u,v)$, we have
\begin{equation*}
\begin{cases}
d \mathcal{K}[  v]  +  v(m(x)-c{\beta-\alpha\over 1-\alpha} v )=0, \\
D \mathcal{P}[v]  +v(M(x)- {\beta-\alpha\over 1-\alpha} v )=0,
\end{cases}
\end{equation*}
where $bc=1$ is used.
The uniqueness of positive steady state to single-species models (\ref{singled}) and (\ref{singleD}) implies that
$$
u_d= c{\beta-\alpha\over 1-\alpha} v,\ \ \ v_D = {\beta-\alpha\over 1-\alpha} v.
$$
Therefore, $bc=1$, $b u_d= v_D$ and all the positive steady states of (\ref{original}) consist of $(su_d, (1-s)v_D)$, $0<s<1$.
\end{proof}

Now we complete the proof of Theorem \ref{thm-main}  on the basis of Propositions \ref{prop-localstability} and \ref{prop-steadystates}.

\begin{proof}[Proof of Theorem \ref{thm-main}]
(i) According to Lemma \ref{lm-signs}, in this case, $\mu_{(u_d,0)}>0$,  $\nu_{(0, v_D)}>0$. Thus thanks to Theorems \ref{thm-monotone} and Lemma \ref{lm-continuous},    one sees that the system (\ref{original}) admits a   positive steady state $(u,v)\in \mathbb X\times \mathbb X$.

Again due to Theorems \ref{thm-monotone}, it suffices to verify the uniqueness of positive steady states.  Suppose that this is not true. Let $(u^*, v^*)$ denote a positive steady state of (\ref{original}) different from $(u,v)$. By Lemma \ref{lm-upper-lower}, there exist an upper solution $(\tilde{u}_0, \tilde{v}_0)$ and a lower solution $(\underline{u}_0, \underline{v}_0)$ of (\ref{original}) such that
$$
(\underline{u}_0, \underline{v}_0) <_c (u,v),\ (u^*, v^*) <_c (\tilde{u}_0, \tilde{v}_0).
$$
Then according to Lemma \ref{lm-monotone}, one sees that the solution of (\ref{original}) with initial value $(\underline{u}_0, \underline{v}_0)$ increases to a positive steady state of (\ref{original}) in $L^{\infty}(\Omega)\times L^{\infty}(\Omega)$, denoted by $(u_1,v_1)$, while the solution of (\ref{original}) with initial value $(\tilde{u}_0, \tilde{v}_0)$ decreases to a positive steady state of (\ref{original}) in $L^{\infty}(\Omega)\times L^{\infty}(\Omega)$, denoted by $(u_2,v_2)$. Thanks to Lemma \ref{lm-continuous}, one has $(u_1,v_1), (u_2,v_2) \in \mathbb X\times \mathbb X.$ Moreover, by comparison principle, it is routine to show that
$$
(u_1,v_1) \leq_c (u,v),\ (u^*, v^*) \leq_c (u_2,v_2).
$$
Therefore, Propositions \ref{prop-localstability} and \ref{prop-steadystates} indicate that
$$
(u_1,v_1) = (u,v)= (u^*, v^*) = (u_2,v_2).
$$
This is a contradiction.

(ii)  According to Theorem \ref{thm-monotone}, to prove that $(0,v_D)$ is globally asymptotically stable,  it suffices to show that (\ref{original}) admits no positive steady states.
    Suppose that (\ref{original}) admits a positive steady state $(u,v)$, i.e., $(u,v)$ satisfies
    \begin{equation*}
     \begin{cases}
    d \mathcal{K}[u]  +u(m(x)-u- cv)=0, \\
    D \mathcal{P}[v]  +v(M(x)-bu- v)=0.
    \end{cases}
    \end{equation*}
    Denote $(u^*, v^*)= (0,v_D)$ and set $w= u- u^*= u>0$, $z=v-v^*<0$. Similar to the computation of (\ref{pf-key}), one has
    \begin{eqnarray*}
 &&  \int_{\Omega}(bu+z)z^2 dx = \int_{\Omega}(bw+z)z^2 dx\cr
&=&  {D\over 2}\int_{\Omega}\int_{\Omega} p(x,y)\left[ v^*(x)v(y)-v(x)v^*(y) \right]^2\left(  \frac{1}{v(x) v(y) } -\frac{1}{v^*(x)v^*(y) }  \right) dy dx \geq  0.
    \end{eqnarray*}
    However,
    \begin{eqnarray*}
    0 &\geq& \nu_{(0, v_D)}= \sup_{0 \neq \phi\in L^2} \frac{\int_{\Omega} \left( d\phi \mathcal{K}[\phi]+[m(x)-cv_D]\phi^2 \right) dx }{\int_{\Omega} \phi^2 dx}\\
    &\geq &  \frac{\int_{\Omega} \left( d u \mathcal{K}[ u ]+[m(x)-cv_D]u^2 \right) dx }{\int_{\Omega} u^2 dx}\\
    &=&\frac{\int_{\Omega} \left( -[m(x)-u-cv]u^2+[m(x)-cv_D]u^2 \right) dx }{\int_{\Omega} u^2 dx}\\
    &=& \frac{\int_{\Omega} (u+cz)u^2   dx }{\int_{\Omega} u^2 dx}.
    \end{eqnarray*}
    Putting together the above two inequalities:
    \begin{equation}\label{important-2}
    \int_{\Omega}(bu+z)z^2 dx\geq 0,\ \ \ \int_{\Omega} (u+cz)u^2   dx \leq 0.
    \end{equation}
    similar to (\ref{pf-keyrelation}), we obtain
    $$
    \int_{\Omega} (u+cz)^2(cz-u)dx\geq 0,
    $$
    where $0<bc\leq 1$ is used. Hence $u+cz =0 $ in $\bar\Omega$ and all the previous inequalities should be equalities. In particular, $bc=1$ and $bu+z =0$. Note that $bu+z =0$ means $bu+v = v_D$. Then based on the equations satisfied by $v$ and $v_D$ respectively, it is routine to show that $v= \alpha v_D$, where $0<\alpha<1$. Thus, $u= c(1-\alpha)v_D$. Then plugging  $v= \alpha v_D$ and $u= c(1-\alpha)v_D$ into the equation satisfied by $u$, we have
    $$
    d c(1-\alpha) \mathcal{K}[v_D]  +c(1-\alpha)v_D (m(x)-c v_D)=0,
    $$
    which indicates that $u_d = c v_D$, i.e., $bu_d= v_D$. This yields a contradiction due to Proposition \ref{prop-localstability}.

(iii) is similar to the proof of case (ii), thus the details are omitted.
\end{proof}


\section{Proof of Theorem \ref{thm-main-2}}
Throughout this section, let $(u(x,t),v(x,t))$ denote a solution of the system (\ref{original}).
First of all, thanks to Proposition \ref{prop-localstability}, if both $(u_d, 0)$ and $(0,v_D)$ are locally stable or neutrally stable, then $bc=1$, $bu_d = v_D$ and thus it is routine to verify that   (\ref{original}) has a continuum of steady states $\{(su_d, (1-s)v_D),\ 0\leq s\leq 1\}$.
It remains to demonstrate the global convergence of solutions to the system (\ref{original}) for any nonnegative initial data $(u_0, v_0) \not\equiv (0,0)$.
The proof of  this part is quite involved and complicated.

Let us add some explanations here for the convenience of readers. If either $u_0 \equiv 0$ or $v_0  \equiv 0$, then
(\ref{original}) is reduced to a single-species model and thus it follows that the corresponding solution $(u(x,t), v(x,t))$ approaches to $(0, v_D)$ or $(u_d,0)$ respectively in $\mathbb X \times \mathbb X$. Now only consider initial data $(u_0, v_0) \in \mathbb X_{++} \times \mathbb X_{++}$. By comparison principle, we have $u(x,t)>0$ and $v(x,t)>0$ in $\bar\Omega$ for $t>0$.    Hence, for the rest of the proof, assume that  $u_0 > 0$, $v_0 > 0$ in $\bar\Omega$ and consider three cases separately:
\begin{itemize}
\item[{\it Case I:}] $u(x,t)$ does not weakly converge to zero  in $L^2(\Omega)$;
\item[{\it Case II:}] $v(x,t)$ does not weakly converge to zero  in $L^2(\Omega)$;
\item[{\it Case III:}] both $u(x,t)$ and $v(x,t)$ weakly converge to zero  in $L^2(\Omega)$.
\end{itemize}

The following property indicates how to initiate the proofs of {\it Cases I} and {\it II}.

\begin{prop}\label{prop-weakL2-nonzero}
Assume that \textbf{(C1)}, \textbf{(C2)}, \textbf{(C3)} hold.
\begin{itemize}
\item[(i)] If Case I holds, 
then there exists $T_1>0$ such that $v(x,t)<v_D(x)$ in $\bar\Omega$ for $t\geq T_1$.
\item[(ii)] If Case II holds, 
then there exists $T_2>0$ such that $u(x,t)<u_d(x)$ in $\bar\Omega$ for $t\geq T_2$.
\end{itemize}
\end{prop}

We prepare a lemma first, which is crucial in the proof of Proposition \ref{prop-weakL2-nonzero}.

\begin{lem}\label{lem-tricky}
Let $\Omega$ denote a bounded domain in $\mathbb R^n$. Assume that $u(\cdot,t)\in L^{\infty}(\Omega)$, $t \geq 0$ satisfies
$$
u_t(x,t) \geq \delta \int_{\Omega \cap B_r(x)} u(y,t) dy,\ 
\textrm{ and } u(x,t)\geq 0 \textrm{ for } t\geq 0,
$$
where $r>0, \delta>0$. Then for any $t_0\geq 0$, $0<t<1$, there exist $\alpha>0$ and $A_0=A_0(\Omega)$ such that
$$
u(x,t_0+t) \geq A_0 t^{\alpha} \int_{\Omega} u(x,t_0) dx\ \ \textrm{in}\ \Omega.
$$
\end{lem}

\begin{proof}
W.l.o.g., assume that $t_0=0$. Note that it is obvious if $u(x,0)\equiv 0$. Now suppose that $u(x,0)\not\equiv 0$ and let $a= \int_{\Omega} u_0(x)dx >0$. Since $\Omega$ is bounded, there exist $x_j \in \mathbb R^n$, $1\leq j\leq J$ such that
$$
\Omega \subset\subset \bigcup_{1\leq j\leq J} B_j, \textrm{ with } B_j \triangleq B_{r/4} (x_j)= \{ x\in \mathbb R^n, \ |x-x_j| < r/4  \}.
$$

W.l.o.g, assume $\sigma = \min \{ |B_j\bigcap \Omega|,\ 1\leq j\leq J \}>0$, $\int_{\Omega  \bigcap B_1} u(y,t) dy \geq a/J$.
and $B_{j+1} \bigcap B_j \neq \varnothing$, $1\leq j\leq J-1$.

Now first  for any $x\in B_1$,
$$
u_t(x,t) \geq \delta \int_{\Omega \bigcap B_r(x)} u(y,t) dy \geq \delta \int_{\Omega \bigcap B_1} u(y,t) dy \geq \delta {a\over J}.
$$
Thus for $x\in B_1$, $t>0$,
\begin{equation}\label{pf-lem-tricky-1}
u(x,t)  \geq \delta {a\over J} t.
\end{equation}

Secondly,  for any $x\in B_2$, it follows that
$$
u_t(x,t) \geq \delta \int_{\Omega \bigcap B_r(x)} u(y,t) dy \geq \delta \int_{\Omega \bigcap B_1} u(y,t) dy \geq \delta {a\over J}.
$$
Thus for $x\in B_2$, $t>0$,
\begin{equation}\label{pf-lem-tricky-2}
u(x,t)  \geq \delta {a\over J} t.
\end{equation}

Next,   for any $x\in B_3$, by (\ref{pf-lem-tricky-2}), one sees that
\begin{eqnarray*}
u_t(x,t) &\geq& \delta \int_{\Omega \bigcap B_r(x)} u(y,t) dy \geq \delta \int_{\Omega \bigcap B_2} u(y,t) dy \\
&\geq& \delta \sigma \delta {a\over J} t = \sigma \delta^2 {a\over J} t.
\end{eqnarray*}
Hence for $x\in B_3$, $t>0$,
$$
u(x,t) \geq \sigma \delta^2 {a\over J} {t^2\over 2}.
$$
This step can be repeated and we have, for $x\in B_j$, $3\leq j\leq J$, $t>0$,
$$
u(x,t) \geq \sigma^{j-2} \delta^{j-1} {a\over J} {t^{j-1}\over (j-1)!}.
$$
Therefore, together with (\ref{pf-lem-tricky-1}) and (\ref{pf-lem-tricky-2}), one sees that, for $x\in\bar\Omega$, $0<t<1$
$$
u(x,t) \geq c_0 t^{J-1} \int_{\Omega} u(x,0) dx,
$$
where
$$
A_0= \min \{ 1,\sigma, \sigma^{J-2} \}  \min\{ \delta, \delta^{J-1} \} {1\over J!}.
$$
The lemma is proved by choosing $\alpha =J-1.$
\end{proof}

\begin{proof}[Proof of Proposition \ref{prop-weakL2-nonzero}]
Assume that {\it Case I}  happens, i.e.,  $u(\cdot, t) \not \rightharpoonup 0$ in $L^2(\Omega)$ as $t  \rightarrow \infty$, then there exist a constant $a_0>0$ and
a sequence $\{ t_j  \}_{j\geq 1}$ with $t_j \rightarrow \infty$ as $j\rightarrow \infty$ such that
\begin{equation}\label{int-lowerbd}
\int_{\Omega} u(x,t_j)dx > a_0\ \ \textrm{for all } j\geq 1.
\end{equation}

First of all, we will derive an uniform lower bound for $u$ in certain time intervals.
According to assumption \textbf{(C2)}, there exist $r_1>0$, $\delta_1>0$ such that $k(x,y)\geq \delta_1$ if $|x-y|\leq r_1$.
Then one sees that
\begin{eqnarray}\label{rough-inq-eqn-u}
u_t &=& d\int_{\Omega}  k(x,y) u(y,t) dy +u(m(x)- d a_d(x) -u-cv)\cr
    &\geq & d \delta_1 \int_{\Omega \bigcap B_{r_1}(x)} u(y,t) dy -A_1 u,
\end{eqnarray}
where
\begin{equation}\label{c1}
A_1 = \sup_{x\in\Omega,\ t>0}\left | m(x)-d a_d(x) -u(x,t)-cv(x,t)\right |.
\end{equation}
Let $U=e^{A_1 t} u$ and it follows
that
$$
U_t \geq d \delta_1 \int_{\Omega \bigcap B_{r_1}(x)} U(y,t) dy.
$$
Thus Lemma \ref{lem-tricky} can be applied to induce that there exist $\alpha>0$, $A_0=A_0(\Omega)$ such that
$$
U (x,t_j+{1\over 2}) \geq A_0 \left({1\over 2}\right)^{\alpha} \int_{\Omega} U(x,t_j) dx\ \ \textrm{in}\ \bar\Omega, \ j\geq 1,
$$
which, by (\ref{int-lowerbd}), implies a crucial estimate:
\begin{eqnarray}\label{app-trickylm}
u (x,t_j+{1\over 2}) &\geq & A_0 e^{-{1\over 2}A_1} \left({1\over 2}\right)^{\alpha} \int_{\Omega} u(x,t_j) dx \cr
&\geq & A_0 e^{-{1\over 2}A_1} \left({1\over 2}\right)^{\alpha} a_0 \doteq A_2\ \ \textrm{in}\ \bar\Omega, \ j\geq 1.
\end{eqnarray}

Thus, thanks to (\ref{app-trickylm}), we have the following estimate for $t> t_j+{1\over 2}$
\begin{eqnarray*}\label{dispersal-lowerbd}
&& d\int_{\Omega}  k(x,y) u(y,t) dy\cr
&=& d\int_{\Omega}  k(x,y) \int_{t_j+{1\over 2}}^t u_{\tau}(y,\tau) d\tau dy +d\int_{\Omega}  k(x,y) u(y,t_j+{1\over 2}) dy\cr
&\geq &  -d \| k(\cdot,\cdot)\|_{L^{\infty}(\Omega\times \Omega)}\| u_t(\cdot, t) \|_{L^{\infty}(\Omega)} (t-t_j- {1\over 2} ) + d A_2\int_{\Omega}  k(x,y) dy.
\end{eqnarray*}
It is easy to see that $\min_{x\in\bar\Omega} \int_{\Omega}  k(x,y) dy >0$ since $\int_{\Omega}  k(x,y) dy \in \mathbb X$.
Denote
$$
\delta_2 =d  A_2\min_{x\in\bar\Omega} \int_{\Omega}  k(x,y) dy .
$$
Also, it is easy to verify that $\| u_t(\cdot, t) \|_{L^{\infty}(\Omega)}$ has an upper bound independent of $t\geq 0$. Hence, there exists $\epsilon_1>0$ such that for any $j\geq 1$, such that
$$
d\int_{\Omega}  k(x,y) u(y,t) dy \geq \delta_2/ 2\ \ \textrm{in}\ \bar\Omega,
$$
which yields that for $t\in [t_j+{1\over 2}, t_j+{1\over 2} +\epsilon_1]$, $x\in \bar\Omega$, $j\geq 1$,
\begin{eqnarray*}
u_t(x,t) &=& d\int_{\Omega}  k(x,y) u(y,t) dy +u(m(x)- a_d(x) -u-cv) \geq  \delta_2/ 2 -A_1 u(x,t),
\end{eqnarray*}
where $A_1$ is determined in (\ref{c1}). Direct computation gives that
$$
u(x,t) \geq {\delta_2 \over 2A_1} \left(  1- e^{-A_1 (t-t_j - {1\over 2} )}         \right)\ \ \textrm{for}\  x\in \bar\Omega,\  t\in [t_j+{1\over 2}, t_j+{1\over 2} +\epsilon_1],
\ j\geq 1.
$$
Therefore, we reach the conclusion that
\begin{equation}\label{pf-lowerbd-u}
u(x,t) \geq A_3\ \ \textrm{for}\  x\in \bar\Omega,\  t\in [t_j+{1\over 2}+ {\epsilon_1\over 2}, t_j+{1\over 2} +\epsilon_1],
\ j\geq 1,
\end{equation}
where
$$
A_3 = {\delta_2 \over 2A_1} \left(  1- e^{-A_1 \epsilon_1/2}\right) >0.
$$

Now we are ready to derive the desired estimates for $v(x,t)$. Note that for single-species model (\ref{singleD}), for any given   initial data in $\mathbb X_{++}$, the corresponding solution $V(\cdot, t) \rightarrow v_D$ in  $\mathbb X$ as
$t\rightarrow \infty.$
Thus, thanks to comparison principle, it routine to verify that there exist  sequences  $\{ h_j\}$ with $h_j>0$, and $\lim_{j\rightarrow \infty}h_j =0$   such that
\begin{equation}\label{pf-upperbd-v}
v(x,t) \leq (1+h_j) v_D(x) \ \ \textrm{in} \  \bar\Omega\ \   \textrm{for}  \ t\geq t_j.
\end{equation}

Notice that to complete the proof, by comparison principle, it suffices to show the existence of $T_1$ such that $v(x,t)<v_D(x)$ in $\bar\Omega$ at $t = T_1$.
Indeed we will prove   that {\it $v(x,t_j+{1\over 2}+\epsilon_1)<v_D(x)$ in $\bar\Omega$ for $j$ large.}

Fix $x\in \bar\Omega$. Suppose that
\begin{equation}\label{pf-prop-v}
  v(x,t) \geq v_D(x)\ \ \ \textrm{for}\ t\in [t_j+{1\over 2}+ {\epsilon_1\over 2}, t_j+{1\over 2}+\epsilon_1],
\end{equation}
which,  by (\ref{pf-lowerbd-u}) and (\ref{pf-upperbd-v}), yields that for $t\in [t_j+{1\over 2}+ {\epsilon_1\over 2}, t_j+{1\over 2}+\epsilon_1]$
\begin{eqnarray}\label{pf-prop-nonloca1}
  v_t(x,  t) &=& D\int_{\Omega}  p(x,y) v(y,  t) dy +v(x,  t)(M(x)- D a_D(x) -b u(x,  t)-v(x,  t)) \cr
    &\leq &  D\int_{\Omega}  p(x,y) v(y,  t) dy +v(x,  t)(M(x)- D a_D(x) -b u(x,  t)-v_D(x))\cr
    &= & D\int_{\Omega}  p(x,y) (1+h_j) v_D(y) dy +v_D(x)(M(x)- D a_D(x) -v_D(x))\cr
    && + (v(x,t)-v_D(x)) (M(x)- D a_D(x)  -b u(x,  t)  -v_D(x))  - b u(x,t)v_D(x)\cr
    &=& O(h_j)  - b A_3  v_D(x) \leq  - A_4,
\end{eqnarray}
for $j $ sufficiently large, where
$$
A_4 = {1\over 2} b A_3 \min_{ \bar\Omega} v_D >0.
$$
Hence
$$
v(x,t_j+{1\over 2}+\epsilon_1) \leq v(x, t_j+{1\over 2}+ {\epsilon_1\over 2}) -{1\over 2} \epsilon_1 A_4\leq (1+h_j) v_D(x)-{1\over 2} \epsilon_1 A_4 <v_D(x)
$$
provided that $j$ is large enough, which contradicts to (\ref{pf-prop-v}).


Therefore, if $j$ is sufficiently large, there exists $s= s(x)\in [t_j+{1\over 2}+ {\epsilon_1\over 2}, t_j+{1\over 2}+\epsilon_1]$ such that $v(x,s)<v_D(x)$. Note that $s$ depends on the choice of $x$ and in fact we need find a moment which is independent of $x\in \bar\Omega$.

We will show that {\it if $j$ is large enough, $v(x,t)<v_D(x)$ for $t\in  [s(x), t_j+{1\over 2}+\epsilon_1].$}
Otherwise, there exists $\tilde t = \tilde t (x)\in (s(x), t_j+{1\over 2}+\epsilon_1]$ such that $v(x,\tilde t)=v_D(x)$ and $v(x,  t)<v_D(x)$ for $t\in (s(x), \tilde t)$. Then, by (\ref{pf-lowerbd-u}) and (\ref{pf-upperbd-v}), it follows that
\begin{eqnarray}\label{pf-prop-nonlocal2}
0\leq v_t(x,\tilde t) &=& D\int_{\Omega}  p(x,y) v(y,\tilde t) dy +v(x,\tilde t)(M(x)- D a_D(x) -b u(x,\tilde t)-v(x,\tilde t)) \cr
    &\leq & D\int_{\Omega}  p(x,y) (1+h_j) v_D(y) dy +v_D(x)(M(x)- D a_D(x) -b A_3-v_D(x))\cr
    &=& h_j  D\int_{\Omega}  p(x,y) v_D(y) dy - b A_3  v_D(x)<0
\end{eqnarray}
for $j$ large. This is a contradiction. Hence, in particular,   $v(x,t_j+{1\over 2}+\epsilon_1)<v_D(x)$ for $j$ sufficiently large.

The proof of   (i) is complete and (ii) can be proved in the same way.
\end{proof}

Now, we continue the proof for {\it Case I}.  With the help of Proposition \ref{prop-weakL2-nonzero}(i), w.l.o.g., we could assume that
$u_0>0$, $0<v_0< v_D$ in $\bar\Omega$. Define
$$
\theta (t) = \sup \{ \theta\ | \ u(x,t)>\theta u_d(x), v(x,t)<(1-\theta ) v_D(x) \ \textrm{in}\ \bar\Omega  \}.
$$
It is obvious that $0< \theta(0) <1$, $\theta(t)$ is increasing in $t$ due to comparison principle. Denote
$$
\theta_* = \lim_{t\rightarrow \infty} \theta (t) \leq 1.
$$

Assume that $\theta_* = 1$.     For $v(x,t) $, since $v(x,t) \leq (1-\theta(t) ) v_D(x)$ in $\bar\Omega$, it is obvious that
$$
v(\cdot, t) \rightarrow 0\  \textrm{in}\  \mathbb X \ \ \textrm{as}\   t \rightarrow \infty.
$$
For $u(x,t) $, compared with the solution $U(x,t)$ of single-species model (\ref{singled}) with initial data $u_0\in \mathbb X_{++}$, one sees that
$u(x,t) \leq   U(x,t).$ Thus it follows from Theorem \ref{thm-single} and the definition of $\theta(t)$ that
$$
u(\cdot, t) \rightarrow u_d(\cdot)\  \textrm{in}\  \mathbb X \ \ \textrm{as}\   t \rightarrow \infty,
$$

It remains to consider  $\theta_* < 1$. For clarity,  the proof of this situation will be divided into three steps.

\noindent{\it Step 1.}  We claim that {\it there exists a subsequence of $(u(\cdot, t), v(\cdot, t))$, which converges to $(\alpha_1 u_d, (1-\alpha_1 ) v_D)$ in $L^2(\Omega) \times L^2(\Omega)$, where $\alpha_1\in [0,1]$ .}

Fix $0<s_1<\theta (0)$,  let $(u^*, v^*) = (s_1 u_d, (1-s_1) v_D )$ and set
$$
w(x,t) = u(x,t)- u^*(x),\ z(x,t) =v(x,t)- v^*(x).
$$
Recall that $(u,v)$ satisfies
\begin{equation*}
\begin{cases}
u_t= d \mathcal{K}[u]  +u(m(x)- u- c v) &\textrm{in } \Omega\times[0,\infty),\\
v_t= D \mathcal{P}[v]  +v(M(x)-b u-  v) &\textrm{in } \Omega\times[0,\infty),
\end{cases}
\end{equation*}
and $(u^*, v^*)$ satisfies
\begin{equation*}
\begin{cases}
    d \mathcal{K}[u^*]  +u^*(m(x)-u^*- cv^*)=0, \\
    D \mathcal{P}[v^*]  +v^*(M(x)-bu^*- v^*)=0.
\end{cases}
\end{equation*}
Thus using the equations satisfied by $u$ and $u^*$, one has
$$
d\left( u^* \mathcal{K}[u] - u \mathcal{K}[u^*] \right)  = u^* u_t + uu^* (w+cz)
$$
This yields that
\begin{equation*}
d\int_{\Omega}\left(-u \mathcal{K}[u^*] + u^* \mathcal{K}[u]  \right) \frac{w^2}{u u^* } dx =  \int_{\Omega} \left( {u_t\over u}w^2 + (w+cz)w^2 \right )dx.
\end{equation*}
Same as the estimates of the left hand side of (\ref{pf-nonpositive})  in the proof of Proposition \ref{prop-steadystates}, we have
\begin{equation}\label{pf-int-w}
\int_{\Omega} \left( {u_t\over u}w^2 + (w+cz)w^2 \right )dx \leq 0.
\end{equation}
Similarly, using the equations satisfied by $v$ and $v^*$, we obtain
\begin{equation}\label{pf-int-z}
\int_{\Omega} \left( {v_t\over v}z^2 + (bw+ z)z^2 \right )dx \geq 0.
\end{equation}

Then (\ref{pf-int-w}), (\ref{pf-int-z}) and $bc=1$ imply that
\begin{eqnarray}\label{pf-int-wz}
&& c^3 \int_{\Omega} {v_t\over v} z^2dx - \int_{\Omega} {u_t\over u} w^2 dx \geq   \int_{\Omega} \left[-c^3(bw+ z)z^2 +  (w+cz)w^2 \right]dx\cr
&=&  \int_{\Omega}   (w+cz)^2 (w - cz)  dx.
\end{eqnarray}
Note that
$$
w- cz = u- u^* -c (v-v^*) \geq (\theta(t) -s_1) u_d + c (\theta(t) -s_1)v_D =2(\theta(t) -s_1) u_d,
$$
since $bc =1$ and $bu_d = v_D$. Denote $C_0 = 2(\theta(0) -s_1) \min_{  \bar\Omega} u_d $. Hence (\ref{pf-int-wz}) becomes
\begin{eqnarray*}
&& \int_{\Omega}   (w+cz)^2 dx\\
&\leq & {1\over C_0} \left( c^3 \int_{\Omega} {v_t\over v} z^2dx - \int_{\Omega} {u_t\over u} w^2 dx  \right)\\
& = & {1\over C_0} \left( c^3 \int_{\Omega} \left(vv_t -2v^* v_t +  (v^*)^2{v_t\over v}  \right)dx - \int_{\Omega}  \left(uu_t -2u^* u_t +  (u^*)^2{u_t\over u}  \right)dx  \right).
\end{eqnarray*}
This implies that
\begin{equation}\label{pf-int-bd}
\int_0^{\infty } \int_{\Omega}   (w+cz)^2 dx dt <\infty.
\end{equation}
Moreover, it is routine to verify that $\int_{\Omega}   (w+cz)^2 dx$ is uniformly continuous in $t$. This, together with (\ref{pf-int-bd}), yields that
\begin{equation}\label{pf-int-0limit}
\lim_{t\rightarrow \infty}\int_{\Omega}   (w+cz)^2 dx =0.
\end{equation}
Again since $bc =1$ and $bu_d = v_D$, $w+cz = u+cv - s_1u_d- c(1-s_1) v_D = u+cv -u_d$.  Hence (\ref{pf-int-0limit}) tells us that
\begin{equation}\label{pf-int-L2}
u(\cdot, t)+cv(\cdot, t) \rightarrow u_d(\cdot) \ \textrm{in}\ L^2(\Omega)\ \textrm{as} \ t\rightarrow \infty.
\end{equation}

Next estimate $\int_{\Omega} u_t^2 dx$ as follows.
\begin{eqnarray*}
\int_{\Omega} u_t^2 dx  &=& \int_{\Omega}  \left( d \mathcal{K}[u] u_t  +u(m(x)- u- c v)u_t \right) dx\\
&=&  {d\over dt} \int_{\Omega} \left(  {1\over 2} d \mathcal{K}[u] u  +{1\over 2} (m-u_d) u^2   \right) dx  +  \int_{\Omega} (u_d - u- cv) uu_t dx\\
&\leq &{d\over dt} \int_{\Omega} \left(  {1\over 2} d \mathcal{K}[u] u  +{1\over 2} (m-u_d) u^2   \right) dx  +  {1\over 2}\int_{\Omega} (u_d - u- cv)^2 u^2  dx  +{1\over 2}\int_{\Omega}  u_t^2 dx,
\end{eqnarray*}
which gives that
\begin{eqnarray*}
&& \int_0^{\infty }\int_{\Omega} u_t^2 dx dt \\
 &\leq & \int_0^{\infty } {d\over dt} \int_{\Omega} \left(    d \mathcal{K}[u] u  +  (m-u_d) u^2   \right) dx dt +   \int_0^{\infty }   \int_{\Omega} (u_d - u- cv)^2 u^2  dx dt  \\
 &<&  \infty
\end{eqnarray*}
thanks to (\ref{pf-int-bd}). Moreover, $\int_{\Omega} u_t^2 dx$ is uniformly continuous in $t$. Thus we obtain that
\begin{equation}\label{pf-int-ut-limit}
\lim_{t\rightarrow \infty}\int_{\Omega}  u_t^2 dx =0.
\end{equation}

Furthermore, by the equation satisfied by $u$:
$$
u_t =d\int_{\Omega}  k(x,y) u(y,t) dy +u(m(x)- d a_d(x) -u-cv),
$$
one has
$$
u(x,t) = \frac{u_t  - d\int_{\Omega}  k(x,y) u(y,t) dy}{m(x)- d a_d(x) -u_d} - \frac{u(u_d-u-cv)}{m(x)- d a_d(x) -u_d},
$$
where, by the equation satisfied by $u_d$,
\begin{equation}\label{pf-int-notzero}
m(x)- d a_d(x) -u_d= - \frac{d\int_{\Omega} k(x,y) u_d(y)dy}{u_d(x)}<0\ \ \textrm{in}\ \bar\Omega.
\end{equation}
Also, notice that for $\phi\in \mathbb X$, the   mapping
$
\phi \rightarrow \int_{\Omega}  k(x,y) \phi (y) dy
$
is compact from $\mathbb X$ to $\mathbb X$.
Thus, there exist  a subsequence $\{u(\cdot,t_j)  \}$, $j\geq 1$, and $\Phi\in \mathbb X$ such that, as $j\rightarrow \infty$,
\begin{equation}\label{pf-int-subseq-Linfty}
\int_{\Omega}  k(x,y) u(y,t_j) dy  \rightarrow \Phi \ \ \textrm{in}\ \  \mathbb X.
\end{equation}
This, together with (\ref{pf-int-L2}) and (\ref{pf-int-ut-limit}), implies that
\begin{equation}\label{pf-int-subseq-L2}
u(\cdot, t_j) \rightarrow \frac{  - d  \Phi (\cdot) }{m(\cdot)- d a_d(\cdot) -u_d(\cdot)}\ \  \textrm{in}\ L^2(\Omega)\ \ \textrm{as} \ j\rightarrow \infty.
\end{equation}
Denote
$$
\tilde{u}= \frac{  - d  \Phi   }{m - d a_d  -u_d } \in \mathbb X.
$$
By (\ref{pf-int-subseq-Linfty}) and (\ref{pf-int-subseq-L2}), we have
$$
d\mathcal{K}[\tilde{u} ] +\tilde{u}(m-u_d)=0,
$$
which implies that there exists $ \alpha_1 \geq 0$ such that $\tilde{u} =  \alpha_1 u_d$ since both $\tilde{u}$ and $u_d$ can be regarded as the  eigenfunctions to the principal eigenvalue zero of the   eigenvalue problem
$d\mathcal{K}[\phi] + (m-u_d) \phi =\mu \phi.$  Thus  (\ref{pf-int-subseq-L2}) becomes
$$
u(\cdot, t_j) \rightarrow  \alpha_1 u_d \ \  \textrm{in}\ L^2(\Omega)\ \ \textrm{as} \ j\rightarrow \infty.
$$
At the end, according to $bc =1$, $bu_d = v_D$ and (\ref{pf-int-L2}),  it is routine to check that $ \alpha_1 \in  [0,1]$ and
$v(\cdot, t_j) \rightarrow (1- \alpha_1) v_D$ in $L^2(\Omega)$ as $j\rightarrow \infty$.

The claim is proved.

\noindent{\it Step 2.}  In this step, we will prove that {\it   $(u(\cdot, t), v(\cdot, t))$ converges  in $L^2(\Omega) \times L^2(\Omega)$.}
Based on the proof in {\it Step 1}, it suffices to show $\theta_* = \alpha_1$. Obviously, $\theta_*  \leq  \alpha_1$. Now suppose that $\theta_* <  \alpha_1$ and  a contradiction will be derived.

According to the definition of $\theta_*$, for any $\delta>0$, there exists $t_{\delta}>0$ such that for $t \geq  t_{\delta}$,
\begin{equation}\label{pf-L2-begin}
u(x,t) > (\theta_* -\delta) u_d(x),\ v(x,t) < (1-\theta_* +\delta ) v_D(x)\ \ \textrm{in} \ \bar\Omega.
\end{equation}
We   claim that {\it there exist $ \epsilon_0>0$, $  \delta_0>0$ and $  j_0\geq 1$ such that for $j\geq   j_0$},
\begin{equation}\label{pf-nonlocal-L2+}
{\it u(x,t_j+  \epsilon_0 ) > (\theta_* + \delta_0) u_d(x),\ v(x,t_j+  \epsilon_0) < (1-\theta_* -  \delta_0 ) v_D(x)\ \ \textrm{in} \ \bar\Omega.}
\end{equation}

Since $u(\cdot, t_j) \rightarrow   \alpha_1 u_d(\cdot)$  in $L^2(\Omega)$ as $j\rightarrow\infty$, it is standard to check that
$$
d  \int_{\Omega}  k(x,y) u(y,t_j) dy  \rightarrow d  \int_{\Omega}  k(x,y)  \alpha_1 u_d(y) dy \ \ \textrm{in}\ \  \mathbb X\ \textrm{as}\ j\rightarrow\infty.
$$
Thus $\theta_* <  \alpha_1$ implies that there exist $\ell_1>0$ and $j_1\geq 1$ such that
for $j\geq j_1$,
\begin{equation}\label{pf-L2-3ell}
d  \int_{\Omega}  k(x,y) u(y,t_j) dy  > d  \int_{\Omega}  k(x,y) \theta_* u_d(y) dy+3\ell_1 \ \ \textrm{in}\ \  \bar\Omega.
\end{equation}
Also, note that $\|u_t(\cdot, t)\|_{L^{\infty}(\Omega)}$ is uniformly bounded in $t$ due to the boundedness of solutions.  It follows from (\ref{pf-L2-3ell}) that there exists $\epsilon_1>0$, independent of $j\geq j_1$, such that for $t\in [t_j,t_j+\epsilon_1]$, $j\geq j_1$,
\begin{equation}\label{pf-L2-2ell}
d  \int_{\Omega}  k(x,y) u(y,t) dy  > d  \int_{\Omega}  k(x,y) \theta_* u_d(y) dy+ 2 \ell_1 \ \ \textrm{in}\ \  \bar\Omega.
\end{equation}
Note that $\epsilon_1$ could be smaller if necessary.

Moreover, there exists $\delta_1>0$ such that for any $0<\delta<\delta_1$, $x\in\bar\Omega$
\begin{equation}\label{pf-L2-reactionterm}
\ell_1 + u(m-a_d -u -cv) > \theta_* u_d \left( m-a_d - \theta_* u_d -c(1-\theta_* )v_D \right),
\end{equation}
as long as $u\in [(\theta_* -\delta) u_d(x), (\theta_* +\delta) u_d(x)]$, $v\in [(1-\theta_* -\delta ) v_D(x), (1-\theta_* +\delta ) v_D(x)]$.

Fix $x\in\bar\Omega$ and $0<\delta<\delta_1$.  Suppose that  if $j\geq j_1$, $t_j\geq t_{\delta}$,  for any $t\in [t_j,t_j+\epsilon_1]$, $u(x,t) \leq (\theta_* +\delta) u_d(x)$.   Then by (\ref{pf-L2-begin}), (\ref{pf-L2-2ell}) and (\ref{pf-L2-reactionterm}),
\begin{eqnarray}\label{pf-thm12-step2-u1}
u_t(x,t) &= & d\int_{\Omega}  k(x,y) u(y,t) dy +u(m(x)- d a_d(x) -u-cv)\cr
   & > & d\int_{\Omega}  k(x,y) u(y,t) dy +u(m(x)- d a_d(x) -u-c(1-\theta_* +\delta ) v_D(x))\cr
   & > &d  \int_{\Omega}  k(x,y) \theta_* u_d(y) dy+ 2 \ell_1 +\theta_* u_d \left( m-a_d - \theta_* u_d -c(1-\theta_* )v_D \right) -\ell_1  \cr
   &=& \ell_1>0,
\end{eqnarray}
which yields that
\begin{eqnarray*}
u(x, t_j+\epsilon_1)  > u(x,t_j) +\ell_1 \epsilon_1 >  (\theta_* -\delta) u_d(x) +\ell_1 \epsilon_1 \geq (\theta_* +\delta) u_d(x)
\end{eqnarray*}
provided that
$$
\delta \leq  {\ell_1\epsilon_1\over 2 \max_{\bar\Omega} u_d}.
$$
This is impossible.
Therefore, given
$$
x\in\bar\Omega,\ \  0 <\delta < \min \left\{ \delta_1,   {\ell_1\epsilon_1\over 2 \max_{\bar\Omega} u_d} \right\},
$$
if $j\geq j_1$, $t_j\geq t_{\delta}$, there exists $\hat{t}_j =\hat{t}_j(x) \in [t_j, t_j+\epsilon_1]$ such that $u(x,\hat{t}_j) > (\theta_* +\delta) u_d(x)$. Note that indeed $\hat{t}_j $ depends on $x$.

Then for all $t\in [\hat{t}_j(x), t_j+\epsilon_1]$, $u(x, t) > (\theta_* +\delta) u_d(x)$ for any $x$ in $\bar\Omega$. Otherwise, if there exist $x^*\in\bar\Omega$ and $t^*_j\in (\hat{t}_j(x^*), t_j+\epsilon_1]$ such that $u(x^*, t^*_j) = (\theta_* +\delta) u_d(x^*)$ and $u(x^*, t ) > (\theta_* +\delta) u_d(x^*)$ for $t\in (\hat{t}_j(x^*), t^*_j) $, then due to (\ref{pf-L2-begin}),  (\ref{pf-L2-2ell}) and (\ref{pf-L2-reactionterm}), we derive that
\begin{eqnarray}\label{pf-thm12-step2-u2}
0 &\geq & u_t(x^*,t^*_j )\cr
 &=& d\int_{\Omega}  k(x,y) u(y,t) dy +u(m(x)- d a_d(x) -u-cv)\Big |_{(x,t)=(x^*, t^*_j)}\cr
&>&   d  \int_{\Omega}  k(x,y) \theta_* u_d(y) dy+ 2 \ell_1 \cr
&& + u(m(x)- d a_d(x) -u-c(1-\theta_* +\delta ) v_D)  \Big |_{(x,t)=(x^*, t^*_j)}\cr
&>&  d  \int_{\Omega}  k(x,y) \theta_* u_d(y) dy+ 2 \ell_1 +\theta_* u_d \left( m-a_d - \theta_* u_d -c(1-\theta_* )v_D \right)-\ell_1\cr
&=& \ell_1>0,
\end{eqnarray}
which is a contradiction.

Thus we have proved that there exist $\ell_1>0$, $j_1\geq 1$, $\epsilon_1>0$ and $\delta_1>0$ such that
$$
u(x, t_j+\epsilon_1) > (\theta_* +\delta) u_d(x)\ \ \textrm{in} \ \bar\Omega,
$$
provided that
$$
0 <\delta < \min \left\{ \delta_1,   {\ell_1\epsilon_1\over 2 \max_{\bar\Omega} u_d} \right\},\  j\geq j_1,\ t_j\geq t_{\delta}.
$$

Similarly, there exist $\ell_2>0$, $j_2\geq 1$, $\epsilon_2>0$ and $\delta_2>0$ such that
$$
v(x,t_j+\epsilon_2) < (1-\theta_* -  \delta ) v_D(x)\ \ \textrm{in} \ \bar\Omega,
$$
provided that
$$
0 <\delta < \min \left\{ \delta_2,   {\ell_2\epsilon_2\over 2 \max_{\bar\Omega} v_D} \right\},\  j\geq j_2,\ t_j\geq t_{\delta}.
$$
Also $\epsilon_2$ could be smaller if necessary.

In summary, choose $ \epsilon_0 = \min \{\epsilon_1, \epsilon_2 \}$ and fix
$$
0< \delta_0 <\min \left\{\delta_1,   {\ell_1  \epsilon_0 \over 2 \max_{\bar\Omega} u_d}, \delta_2,   {\ell_2  \epsilon_0\over 2 \max_{ \bar\Omega} v_D} \right\}
$$
and  $ j_0$ large enough such that   for $j\geq   j_0$, $t_j\geq t_{ \delta_0}$, then for $j\geq   j_0$,
\begin{equation*}
{\it u(x,t_j+  \epsilon_0 ) > (\theta_* + \delta_0) u_d(x),\ v(x,t_j+  \epsilon_0 ) < (1-\theta_* -  \delta_0 ) v_D(x)\ \ \textrm{in} \ \bar\Omega},
\end{equation*}
i.e., (\ref{pf-nonlocal-L2+}) is proved.
This is a contradiction to the definition of $\theta_*$.

Therefore, $\theta_* =  \alpha_1$ and  it follows that $(u(\cdot, t), v(\cdot, t))$ converges to $( \alpha_1 u_d, (1- \alpha_1  ) v_D)$  in $L^2(\Omega) \times L^2(\Omega)$ as $t\rightarrow \infty$.

\noindent{\it Step 3.} We will improve the $L^2(\Omega)\times L^2(\Omega)-$convergence to $\mathbb X\times \mathbb X-$convergence in this step. Define
$$
\eta (t) = \inf \{ \eta\ | \ u(x,t)<\eta u_d(x), v(x,t) > (1-\eta ) v_D(x) \ \textrm{in}\ \bar\Omega  \}.
$$
Obviously, $\eta(t)$ is decreasing in $t$ due to comparison principle. Denote
$
\eta ^* = \lim_{t\rightarrow \infty} \eta (t).
$

Notice that $\theta_* =\alpha_1<1$ immediately yields that $v(x,t)$ does not weakly converge to zero  in $L^2(\Omega)$. Due to Proposition \ref{prop-weakL2-nonzero}(ii), there exists $T_2>0$ such that $u(x,t)<u_d(x)$ in $\bar\Omega$ for $t\geq T_2$.
Hence, for  {\it Case I}, w.l.o.g., assume that
$$
0<u_0 <u_d,\ 0<v_0<v_D\ \ \textrm{in}\ \bar\Omega.
$$
This indicates that  $0 <\theta(0), \eta(0)<1$.

According the definitions of $\theta_*$,  $\eta ^*$ and $\alpha_1$, it is obvious that $\theta_* \leq  \alpha_1 \leq \eta ^*$,  $\theta_*\leq 1$ and $\eta^*\geq 0$.
There are three situations to discuss.
\begin{itemize}
\item {\it $\theta_* =1$.} It has been proved  before {\it Step 1} that
$$
\lim_{t\rightarrow \infty }(u(\cdot, t), v(\cdot, t)) =(   u_d, 0)\ \ \textrm{in}\ \mathbb X \times \mathbb X.
$$
\item {\it $\eta^* =0$.} Similar to the proof when $\theta_* =1$, it follows that
$$
\lim_{t\rightarrow \infty }(u(\cdot, t), v(\cdot, t)) =(  0, v_D)\ \ \textrm{in}\ \mathbb X \times \mathbb X.
$$
\item {\it $\theta_* <1$ and $\eta^* >0$.} According to {\it Steps 1 and 2}, $\theta(0)>0$ and $\theta_* <1$ yield  that $\theta_* =  \alpha_1$. Similarly, it can be proved that $\eta(0)<1$ and $\eta^* >0$ imply that $\eta ^* = \alpha_1$. Hence $\theta_* = \eta ^* = \alpha_1\in (0,1)$ and thus $(u(\cdot, t), v(\cdot, t))$ converges to $( \alpha_1 u_d, (1- \alpha_1  ) v_D)$  in $\mathbb X \times \mathbb X$ as $t\rightarrow \infty$.
\end{itemize}

Therefore, the proof of {\it Case I} is complete. Obviously, {\it Case II} can be proved in the same way.

At the end, let us handle {\it Case III} when both $u(x,t)$ and $v(x,t)$ weakly converge to zero  in $L^2(\Omega)$.
Indeed we will show that {\it Case III} cannot happen. We prepare the following proposition first.
\begin{prop}\label{prop-weakL2-zero}
Assume that \textbf{(C1)}-\textbf{(C3)} hold.
\begin{itemize}
\item[(i)] If $u(\cdot, t) \rightharpoonup 0$ in $L^2(\Omega)$ as $t\rightarrow\infty$,
then there exists $T_3>0$ such that $u(x,t)<u_d(x)$ in $\bar\Omega$ for $t\geq T_3$.
\item[(ii)] If $v(\cdot, t) \rightharpoonup 0$ in $L^2(\Omega)$ as $t\rightarrow\infty$,
then there exists $T_4>0$ such that $v(x,t)<v_D(x)$ in $\bar\Omega$ for $t\geq T_4$.
\end{itemize}
\end{prop}

\begin{proof}
(i) Choose
$$
0<\ell \leq {1\over 2} \min \left\{\min_{ \bar\Omega}  \frac{d  \int_{\Omega}  k(x,y)  u_d(y) dy}{ u_d(x)},\ \min_{\bar\Omega} u_d  \right\}.
$$
It follows from the equations satisfied by $u$ and $u_d$ that
\begin{eqnarray}\label{pf-prop-zero-ut}
u_t(x,t) &= & d\int_{\Omega}  k(x,y) u(y,t) dy +u(m(x)- d a_d(x) -u-cv)\cr
   & = &  d\int_{\Omega}  k(x,y) u(y,t) dy + u\left(  u_d - \frac{d  \int_{\Omega}  k(x,y)  u_d(y) dy}{ u_d(x)} -u-cv  \right)\cr
   &\leq &  d\int_{\Omega}  k(x,y) u(y,t) dy  + u(u_d - 2\ell -u).
\end{eqnarray}

Note that if $u(\cdot, t) \rightharpoonup 0$ in $L^2(\Omega)$ as $t\rightarrow\infty$, then it is routine to check that as $t\rightarrow\infty$,
$$
d\int_{\Omega}  k(\cdot ,y) u(y,t) dy  \rightarrow 0\ \ \textrm{in} \ L^{\infty}( \Omega).
$$
Denote $\epsilon = {\ell\over 4} \min_{\bar\Omega} u_d$. There exists $T_0>0$ such that
for $t\geq T_0$,
\begin{equation}\label{pf-prop-hatT}
d\int_{\Omega}  k(x ,y) u(y,t) dy \leq \epsilon\ \ \textrm{in} \ \bar\Omega.
\end{equation}
Denote $C_1 = \| u(x,t)\|_{L^{\infty}(\bar\Omega\times [0,\infty))}<\infty$. We claim that {\it $u(x,t) < u_d(x)- \ell$ in $\bar\Omega$ for $t\geq T_0+ C_1/\epsilon$.}

Suppose that the claim is not true, i.e., there exist $\hat x\in\bar\Omega$ and $\hat t\geq T_0+ C_1/\epsilon$ such that $u(\hat x,\hat t) \geq u_d(\hat x)- \ell$.

First, fix $x\in\bar\Omega$, we show that  if for some $t_0\geq T_0$, $u(x,t_0) < u_d(x)- \ell$, then $u(x,t) < u_d(x)- \ell$ for $t\geq t_0$. Otherwise, if there exists $t_1>t_0$ such that $u(x,t_1) = u_d(x)- \ell$ and $u(x,t) < u_d(x)- \ell$ for $t_0<t<t_1$, then by (\ref{pf-prop-zero-ut}) and (\ref{pf-prop-hatT}),
\begin{eqnarray*}
0\leq u_t(x,t_1) = &\leq &  d\int_{\Omega}  k(x,y) u(y,t_1) dy  + u(x,t_1)(u_d - 2\ell -u(x,t_1)) \\
&\leq &   \epsilon -\ell (u_d(x)- \ell) \leq {\ell\over 4} \min_{\bar\Omega} u_d- {\ell\over 2} \min_{\bar\Omega} u_d = -\epsilon <0,
\end{eqnarray*}
which is impossible.

Now one sees that for any $t\in [T_0,  T_0+ C_1/\epsilon] $, $u(\hat x,  t) \geq u_d(\hat x)- \ell$. Then by (\ref{pf-prop-zero-ut}) and (\ref{pf-prop-hatT}), when $t\in [T_0,  T_0+ C_1/\epsilon] $,
\begin{eqnarray*}
u_t(\hat x,t) &\leq &  d\int_{\Omega}  k(\hat x,y) u(y,t) dy  + u(\hat x,t)(u_d(\hat x) - 2\ell -u(\hat x, t)) \\
&\leq &  {\ell\over 4} \min_{\bar\Omega} u_d -\ell u(\hat x,t)\leq  {\ell\over 4} \min_{\bar\Omega} u_d -\ell (u_d(\hat x)- \ell) \leq -\epsilon.
\end{eqnarray*}
This gives that
$$
u(\hat x, T_0+ C_1/\epsilon) \leq  u(\hat x, T_0 )  -\epsilon C_1/\epsilon\leq 0,
$$
which contradicts to the positivity of $u$. The claim is proved and (i) follows.

Obviously, (ii) can be proved similarly.
\end{proof}

Thanks to Proposition \ref{prop-weakL2-zero}, w.l.o.g., assume that
$$
0<u_0 <u_d,\ 0<v_0<v_D\ \ \textrm{in}\ \bar\Omega.
$$
and {\it Steps 1, 2 and 3} in the proof of {\it Case I} can be repeated. Thus the solution  $(u,v)$ of (\ref{original}) approaches to  a steady state in $\{(su_d, (1-s)v_D),\ 0\leq s\leq 1\}$ in $\mathbb X \times \mathbb X$.  This is impossible since  $u(\cdot, t) \rightharpoonup 0$ and $v(\cdot, t) \rightharpoonup 0$ in $L^2(\Omega)$ as $t\rightarrow\infty$. Therefore, {\it Case III} cannot happen.

\section{Models with mixed dispersal strategies}
This section is devoted to the proof of Theorem \ref{thm-main-mix}, which is about the system (\ref{general-mix-N}).
The general approaches in handling Theorem \ref{thm-main-mix}(i), (ii) and (iii) are similar to that of Theorem \ref{thm-main}. To avoid being redundant, we only emphasize the places which are different.    If both equations in  (\ref{general-mix-N})   have local dispersals,  then the solution orbits are precompact and thus Theorem \ref{thm-main-mix}(iv) has been established in \cite{HS2006}. However, if local dispersal is only  incorporated into one equations in (\ref{general-mix-N}), additional techniques and adjustments are needed on the basis of the proof of Theorem \ref{thm-main-2}.

First of all, consider the linearized operators of (\ref{general-mix-N})  at $(\hat u_d, 0)$ and $(0, \hat v_D)$. If $\beta=1$, $\mu_{(\hat u_d,0)}$ is defined in the same way as in (\ref{PEV}). If $0\leq \beta<1$, $\mu_{(\hat u_d,0)}$ denotes the principal eigenvalue of  the eigenvalue problem
\begin{equation*}
\begin{cases}
D \left\{\beta \mathcal{P}[\psi] +(1-\beta) \Delta\psi  \right\}+[M(x)-b\hat u_d]\psi =\mu\psi  & \textrm{in}\ \bar\Omega,\\
\partial\psi/\partial \gamma =0  & \textrm{on}\ \partial\Omega.
\end{cases}
\end{equation*}
The definition for $\nu_{(0, \hat v_D)}$ is similar.
Propositions \ref{prop-localstability} and \ref{prop-steadystates} still hold for the system (\ref{general-mix-N}).
\begin{prop}\label{prop-localstability-mixed}
Assume that \textbf{(C1)}-\textbf{(C4)} hold and (\ref{general-condition}) is valid. Then there exist exactly four alternatives as follows.
\begin{itemize}
\item[(i)] $\mu_{(\hat u_d,0)}>0$,  $\nu_{(0, \hat v_D)}>0$;
\item[(ii)] $\mu_{(\hat u_d,0)}>0$,  $\nu_{(0, \hat v_D)}\leq 0$;
\item[(iii)] $\mu_{(\hat u_d,0)}\leq 0$,  $\nu_{(0, \hat v_D)}>0$;
\item[(iv)] $\mu_{(\hat u_d,0)}= \nu_{(0, \hat v_D)}=0$.
\end{itemize}
Moreover,  $(iv)$  holds if and only if $b(x), c(x), b_1(x), c_1(x)$ are constants, $bc=b_1 c_1$ and $b\hat u_d= c_1 \hat v_D$.
\end{prop}

The proof of Proposition \ref{prop-localstability-mixed} is the same as that of Proposition \ref{prop-localstability} and thus we omit the details.

\begin{prop}\label{prop-steadystates-mixed}
Assume that \textbf{(C1)}-\textbf{(C4)} hold and (\ref{general-condition}) is valid. Then the system (\ref{general-mix-N}) admits two strictly ordered continuous positive steady states $(u,v)$ and $(u^*,v^*)$ (that is w.l.o.g., $u>u^*$, $v<v^*$) if and only if $bc=b_1 c_1$, $b \hat u_d=c_1 \hat v_D$. Moreover, all the positive steady states of (\ref{general-mix-N}) consist of $(s\hat u_d, (1-s)\hat v_D)$, $0<s<1$.
\end{prop}

\begin{proof}
Set $w = u- u^*>0$ and $z= v-v^*<0$. Following the proof of Proposition \ref{prop-steadystates}, we only explain how to obtain the following two important inequalities:
\begin{equation}\label{prop-important-general}
\int_{\Omega}\left(b_1(x)w+c(x)z\right)w^2 dx\leq 0,\ \ \ \int_{\Omega}\left(b(x)w+c_1(x)z\right)z^2 dx\geq 0.
\end{equation}

For this purpose, first, similar to (\ref{pf-nonpositive}), it is routine to check that
\begin{eqnarray}\label{pf-3}
&& d\int_{\Omega}\left(-u \left\{\alpha \mathcal{K}[ u^*]+(1-\alpha) \Delta   u^* \right\}  + u^* \left\{\alpha \mathcal{K}[u]+(1-\alpha) \Delta u \right\} \right) \frac{w^2}{u u^*} dx \cr
&=&  \int_{\Omega}(b_1(x)w+c(x)z)w^2 dx.
\end{eqnarray}
Then due to \textbf{(C3)},    the left hand side of  (\ref{pf-3}) is calculated as follows
\begin{eqnarray*}
&&  d\int_{\Omega}\left(-u \left\{\alpha \mathcal{K}[ u^*]+(1-\alpha) \Delta   u^* \right\}  + u^* \left\{\alpha \mathcal{K}[u]+(1-\alpha) \Delta u \right\} \right) \frac{w^2}{u u^*} dx  \\
&=& d\alpha \int_{\Omega}\int_{\Omega} k(x,y)\left[ u^*(x)u(y)-u(x)u^*(y) \right] \frac{(u(x)- u^*(x))^2}{u(x) u^*(x) } dy dx\\
&& + d(1-\alpha) \int_{\Omega} \left( -u \Delta u^*+  u^* \Delta u\right)\frac{(u(x)- u^*(x))^2}{u(x) u^*(x) } dx\\
&=&  {d\over 2}\int_{\Omega}\int_{\Omega} k(x,y)\left[ u^*(x)u(y)-u(x) u^*(y) \right]^2\left(  \frac{1}{u(x) u(y) } -\frac{1}{ u^*(x) u^*(y) }  \right) dy dx\\
&& + d(1-\alpha) \int_{\Omega} | u(x) \nabla  u^*(x) - u^*(x) \nabla u(x)|^2\left(  \frac{1}{u^2(x) } -\frac{1}{ (u^*)^2(x) }  \right)dx\\
&\leq & 0.
\end{eqnarray*}
Thus
$$
\int_{\Omega}\left(b_1(x)w+c(x)z\right)w^2 dx\leq 0,
$$
while the other inequality in (\ref{prop-important-general}) can be handled similarly.

Obviously, since $w>0$ and $z<0$, (\ref{prop-important-general}) implies that
\begin{equation}\label{important-general-bc}
\int_{\Omega}\left([\min_{\bar\Omega}b_1] w+[\max_{\bar\Omega}c]z\right)w^2 dx\leq 0,\ \ \ \int_{\Omega}\left([\max_{\bar\Omega}b]w+[\min_{\bar\Omega}c_1]z\right)z^2 dx\geq 0.
\end{equation}
Now  the arguments after (\ref{important}) can be applied to show that  $b(x), c(x), b_1(x), c_1(x)$ must be constants, $bc=b_1c_1$ and $bu^*_d = c_1v^*_D$.
\end{proof}

Now we are ready to continue the   proof of Theorem \ref{thm-main-mix}.
First, Theorem \ref{thm-main-mix}(i), (ii) and (iii) can be handled by the same approach employed in the proof of Theorem \ref{thm-main}.
Secondly, according to Proposition \ref{prop-localstability-mixed}(iv),  when   both $(\hat u_d, 0)$ and $(0,\hat v_D)$ are locally stable or neutrally stable, then $b(x), c(x), b_1(x), c_1(x)$ must be constants, $bc=b_1c_1$, $b\hat{u}_d = c_1\hat{v}_D$ and the system (\ref{general-mix-N}) has a continuum of steady states $\{(s\hat{u}_d, (1-s)\hat{v}_D),\ 0\leq s\leq 1\}$.
It remains to verify the global convergence of solutions of (\ref{general-mix-N}).

For clarity, we divide it into three cases.

\noindent{\it Case 1: $\alpha= \beta =1$.} This corresponds to the  system (\ref{original}) and has been proved in Theorem \ref{thm-main-2} already.

\noindent{\it Case 2: $0\leq \alpha,\beta <1$, i.e, local dispersals are incorporated into both equations of the system (\ref{general-mix-N}). }  Then the solution orbit  $\{ (u(\cdot, t), v(\cdot, t)) \ |\  t\geq 0 \}$ is precompact in $L^{\infty}(\Omega) \times L^{\infty}(\Omega)$. Moreover, it is standard to verify that $(0,0)$ is locally unstable due to the existence of $\hat u_d$ and $\hat v_D$. Therefore, the conclusion follows from the arguments in the proof of  \cite[Theorem 3]{HS2006}.

\noindent{\it Case 3: $\alpha =1$, $0\leq \beta <1$ or $0\leq \alpha  <1$, $\beta =1$, i.e. local dispersal is only incorporated into one equation of system (\ref{general-mix-N}).} We only prove the case $\alpha =1$, $0\leq \beta <1$, since the other one can be handled in the same way.

The rest of this section is devoted to the proof of {\it Case 3}.  We will mainly follow  the structure of the proof  of Theorem \ref{thm-main-2}. However,     the introduction of local dispersal to only one equation  causes extra obstacles and some new ideas are needed to overcome these difficulties. For clarity, we focus on the following system
\begin{equation}\label{general-local+nonlocal}
\begin{cases}
u_t= d   \mathcal{K} [u]  +u(m(x)-b_1 u- c v) &\textrm{in } \Omega\times[0,\infty),\\
v_t= D \mathcal{P}_{\beta}[v]  +v(M(x)-b u- c_1 v) &\textrm{in } \Omega\times[0,\infty),\\
(1-\beta) \partial v/\partial \gamma=0   &\textrm{on } \partial\Omega,\\
u(x,0)=u_0(x),~v(x,0)=v_0(x)  &\textrm{in } \Omega,
\end{cases}
\end{equation}
where $\mathcal{P}_{\beta}[v]= \beta \mathcal{P} [v]+(1-\beta) \Delta v $. Also, let $(u(x,t),v(x,t))$ denote the corresponding solution.

First of all, assume that $u(x,t)$ does not weakly converge to zero  in $L^2(\Omega)$ and  prepare the following proposition for system (\ref{general-local+nonlocal}), which is parallel to Proposition \ref{prop-weakL2-nonzero}(i). But the proof has to be modified since $v$ satisfies an equation with local dispersal now. To be more specific, the inequalities (\ref{pf-prop-nonloca1}) and (\ref{pf-prop-nonlocal2}) do not hold when local dispersal is incorporated.

\begin{prop}\label{prop-weakL2-nonzero-local+nonlocal}
Assume that \textbf{(C1)}, \textbf{(C2)}, \textbf{(C3)} hold. If $u(x,t)$ does not weakly converge to zero  in $L^2(\Omega)$,
then there exists $T_1>0$ such that $v(x,t)<\hat v_D(x)$ in $\bar\Omega$ for $t\geq T_1$.
\end{prop}

\begin{proof}
Since $u(\cdot, t) \not \rightharpoonup 0$ in $L^2(\Omega)$ as $t  \rightarrow \infty$ and $u(x,t)$ satisfies the equation with nonlocal dispersal only, the arguments in deriving  (\ref{pf-lowerbd-u})   can be applied word by word to indicate that
there exist a constant $B_1>0$, $\varepsilon _1>0$ and
a sequence $\{ \tau_j  \}_{j\geq 1}$ with $\tau_j \rightarrow \infty$ as $j\rightarrow \infty$ such that
\begin{equation}\label{pf-lowerbd-u-local+nonlocal}
u(x,t) \geq B_1\ \ \textrm{for}\  x\in \bar\Omega,\  t\in [\tau_j+{1\over 2}+ { \varepsilon _1\over 2}, \tau_j+{1\over 2} + \varepsilon _1], \ j\geq 1.
\end{equation}
Moreover, comparison principle implies that there exists a sequence $\{ \ell_j \}$ with $\ell_j>0$ and $\lim_{j\rightarrow \infty} \ell _j =0$ such that
\begin{equation}\label{pf-upperbd-v-local+nonlocal}
v(x,t) \leq (1+\ell _j) \hat v_D(x) \ \ \textrm{in} \  \bar\Omega\ \   \textrm{for}  \ t\geq \tau _j.
\end{equation}
Define
$$
V_1(x,t) = \left( 1- \sigma(t-\tau_j - {1\over 2} - { \varepsilon _1\over 2})\right)(1+\ell _j) \hat v_D(x),
$$
where $\sigma>0$ is to be determined later.

For $x\in \bar\Omega,\  t\in [\tau_j+{1\over 2}+ { \varepsilon _1\over 2}, \tau_j+{1\over 2} + \varepsilon _1], \ j\geq 1$, direct computation gives that
\begin{eqnarray*}
&& D \mathcal{P}_{\beta}[V_1]  +V_1(M(x)-b u- c_1 V_1)-\frac{\partial V_1}{\partial t}\\
&\leq & ( 1-  \sigma(t-\tau_j - {1\over 2} - { \varepsilon _1\over 2}))(1+\ell _j)\left( D \mathcal{P}_{\beta}[\hat v_D]  +\hat v_D(M(x)-b B_1- c_1 V_1) \right)  +  \sigma (1+\ell _j) \hat v_D \\
&\leq  &  \left( 1-  O( \sigma \varepsilon _1) \right )(1+\ell _j)\hat v_D \left(-b B_1 + O(\ell_j) + O( \sigma \varepsilon _1)  + O(\sigma\ell_j    \varepsilon_0)\right)+  \sigma (1+\ell _j) \hat v_D \\
&<&0
\end{eqnarray*}
if $\sigma$ is chosen to be small enough and $j$ is large enough.  Note that $\sigma$ is fixed now and we still have the freedom for the choice of $j$. Moreover, it is obvious that (\ref{pf-upperbd-v-local+nonlocal}) implies that
$$
v(x, \tau_j+{1\over 2}+ { \varepsilon _1\over 2})\leq  V_1(x, \tau_j+{1\over 2}+ { \varepsilon _1\over 2}) \ \ \textrm{in}\ \bar\Omega.
$$
Then thanks to comparison principle, it follows that
$$
v(x, \tau_j+{1\over 2}+  \varepsilon _1 )\leq  V_1(x, \tau_j+{1\over 2}+   \varepsilon _1 ) \ \ \textrm{in}\ \bar\Omega.
$$
Furthermore, it is routine to check that
\begin{eqnarray*}
&& V_1(x, \tau_j+{1\over 2}+   \varepsilon _1 ) = \left( 1- \sigma  {\varepsilon _1\over 2} \right)(1+\ell _j) \hat v_D(x)\\
 &\leq & \left(1-  \sigma  {\varepsilon _1\over 2} +\ell_j -\sigma  {\varepsilon _1\over 2}\ell_j  \right) \hat v_D(x)
 \leq  \left(1- {1\over 2} \sigma  {\varepsilon _1\over 2}   \right) \hat v_D(x)\ \ \textrm{in} \ \bar\Omega
\end{eqnarray*}
for $j$ sufficiently large.

The proof is complete.
\end{proof}

Now thanks to Proposition \ref{prop-weakL2-nonzero-local+nonlocal}, w.l.o.g., we could assume that
$u_0>0$, $0<v_0< \hat v_D$ in $\bar\Omega$ and define
$$
\hat\theta (t) = \sup \{ \theta\ | \ u(x,t)>\theta \hat u_d(x), v(x,t)<(1-\theta ) \hat v_D(x) \ \textrm{in}\ \bar\Omega  \}.
$$
Moreover,   $0<\hat \theta(0) <1$, $\hat \theta(t)$ is increasing in $t$ due to comparison principle and denote
$$
\hat \theta_* = \lim_{t\rightarrow \infty} \hat\theta (t) \leq 1.
$$
As explained before {\it Step 1} in Section 4, when $\hat \theta_* = 1$, one has
$$
\lim_{t\rightarrow \infty }(u(\cdot, t), v(\cdot, t)) =(   \hat u_d, 0)\ \ \textrm{in}\ \mathbb X \times \mathbb X.
$$

Let us consider the case that $\hat \theta_* < 1$. To make   the arguments transparent, we discuss it step by step.

\noindent{\it Step 1'.} Considering how (\ref{prop-important-general}) is verified, similar to the arguments in {\it Step 1} in Section 4, we obtain that   there exists a subsequence $\{\tau_j\}_{j\geq 1} $ with $\tau_j \rightarrow \infty$ as $j\rightarrow \infty$ and $\hat\alpha_1\in [0,1]$ such that
$$
\lim_{j\rightarrow \infty }(u(\cdot, \tau_j), v(\cdot, \tau_j)) =(\hat\alpha_1 \hat u_d, (1-\hat\alpha_1 ) \hat v_D)\ \ \textrm{in}\ L^2(\Omega) \times L^2(\Omega).
$$

\noindent{\it Step 2'.}  Similar to {\it Step 2} in Section 4, to prove that $(u(\cdot, t), v(\cdot, t))$ converges  in $L^2(\Omega) \times L^2(\Omega)$, one needs to show that $\hat\theta_* =\hat\alpha_1$.  For this purpose, suppose that $\hat\theta_* < \hat\alpha_1$ and  a contradiction will be derived.

According to the definition of $\hat \theta_*$, for any $\delta>0$, there exists $\tau_{\delta}>0$ such that for $t \geq  \tau_{\delta}$,
\begin{equation}\label{pf-L2-}
u(x,t) > (\hat\theta_* -\delta) \hat u_d(x),\ v(x,t) < (1-\hat \theta_* +\delta )\hat v_D(x)\ \ \textrm{ in} \ \bar\Omega.
\end{equation}
We   claim that {\it there exist $\hat{\varepsilon} >0$, $\hat \delta>0$ and $\hat j\geq 1$ such that for $j\geq \hat j$},
\begin{equation}\label{pf-L2+}
 u(x,\tau_j+ \hat{\varepsilon} ) > (\hat{\theta}_* +\hat\delta) \hat u_d(x),\ v(x,\tau_j+ \hat\varepsilon ) < (1-\hat{\theta} _* - \hat\delta ) \hat v_D(x)\ \ \textrm{{\it in}} \ \bar\Omega.
\end{equation}

Obviously, for $u(x,t)$ in (\ref{general-local+nonlocal}), the same arguments in {\it Step 2} in Section 4 can be applied to show that there exist $\ell_1>0$, $j_1\geq 1$, $\varepsilon_1>0$ and $\delta_1>0$ such that
$$
u(x, \tau_j+\varepsilon_1) > (\hat{\theta}_* +\delta) \hat u_d(x)\ \ \textrm{in} \ \bar\Omega,
$$
provided that
$$
0 <\delta < \min \left\{ \delta_1,   {\ell_1\varepsilon_1\over 2 \max_{\bar\Omega}\hat u_d} \right\},\  j\geq j_1,\ \tau_j\geq \tau_{\delta}.
$$
Here fix $\delta= \delta_2>0$ satisfying the above inequality.

However, the arguments for $u(x,t)$ can not be applied to handle $v(x,t)$, since (\ref{pf-L2-3ell}), (\ref{pf-L2-2ell}), (\ref{pf-thm12-step2-u1}) and (\ref{pf-thm12-step2-u2}) are not valid when local dispersal is incorporated. The idea in the proof of Proposition \ref{prop-weakL2-nonzero-local+nonlocal} is borrowed here. We include the details for the convenience of readers.

Notice that  $\| u_t(\cdot, t) \|_{L^{\infty}(\Omega)}$ has an upper bound independent of $t\geq 0$, thus there exists $\varepsilon_2>0$ such that for $t\in [\tau_j+\varepsilon_1, \tau_j+\varepsilon_1+\varepsilon_2]$, $\tau_j\geq \tau_{\delta_2}$
$$
u(x, t) > (\hat{\theta}_* +{\delta_2 \over 2}) \hat u_d(x)\ \ \textrm{in} \ \bar\Omega.
$$

Define
$$
V_2(x,t) = \left( 1- \sigma_1(t-\tau_j -  \varepsilon _1 )\right)(1-\hat \theta_* +\delta ) \hat v_D(x),
$$
where $\sigma_1$ and $\delta $  are to be determined later.
For $x\in \bar\Omega,\  t\in [\tau_j+\varepsilon_1, \tau_j+\varepsilon_1+\varepsilon_2], \ \tau_j\geq \tau_{\delta}$, direct computation gives that
\begin{eqnarray*}
&& D \mathcal{P}_{\beta}[V_2]  +V_2(M(x)-b u- c_1 V_2)-\frac{\partial V_2}{\partial t}\\
&\leq &\left( 1- \sigma_1(t-\tau_j -  \varepsilon _1 )\right)(1-\hat \theta_* +\delta )\left( D \mathcal{P}_{\beta}[\hat v_D]  +\hat v_D(M(x)-b (\hat{\theta}_* +{\delta_2\over 2}) \hat u_d- c_1 V_2) \right) \\
&& +  \sigma_1 (1-\hat \theta_* +\delta ) \hat v_D \\
&\leq  &  \left( 1-  O( \sigma_1 \varepsilon _2) \right)(1-\hat \theta_* +\delta )\hat v_D \left(-c_1 {\delta_2\over 2} -c_1\delta  +  c_1 \sigma_1(t-\tau_j -  \varepsilon _1 )(1-\hat \theta_* +\delta )\right) \hat v_D \\
&& +  \sigma_1 (1-\hat \theta_* +\delta ) \hat v_D \\
&<&0
\end{eqnarray*}
provided that $\sigma_1>0$ is sufficiently small and fixed.
Moreover, (\ref{pf-L2-}) indicates that for $\tau_j\geq \tau_{\delta}$,
$$
v(x, \tau_j+\varepsilon_1) \leq V_2(x,\tau_j+\varepsilon_1)\ \ \textrm{in}\ \bar\Omega.
$$
Then, due to comparison principle, we have
\begin{eqnarray*}
 v(x, \tau_j+\varepsilon_1+ \varepsilon_2) &\leq &  V_2(x,\tau_j+\varepsilon_1+ \varepsilon_2)\\
&=& \left( 1- \hat \theta_* - \sigma_1  \varepsilon _2  (1-\hat \theta_* ) +\delta - \sigma_1  \varepsilon _2\delta \right)\hat v_D(x)\\
&\leq &\left( 1- \hat \theta_* - {1\over 2}\sigma_1  \varepsilon _2  (1-\hat \theta_* )   \right)\hat v_D(x)
\end{eqnarray*}
by choosing $\delta \leq  {1\over 2}\sigma_1  \varepsilon _2  (1-\hat \theta_* )$.

In summary, set
$$
\hat \varepsilon = \varepsilon_1 + \varepsilon_2,\ \hat{\delta}= \min \left\{{\delta_2 \over 2},{1\over 2}\sigma_1  \varepsilon _2  (1-\hat \theta_* ) \right\}
$$
and choose $\hat j$ such that for  $j\geq \hat j$, $\tau_j\geq \tau_{\hat\delta}$.
The claim is proved. This contradicts to the definition of $\hat \theta_*$. Therefore,
$\hat\theta_* =\hat\alpha_1$ and thus
\begin{equation}\label{pf-thm13-L2}
\lim_{t\rightarrow \infty }(u(\cdot, t), v(\cdot, t)) =(\hat\alpha_1 \hat u_d, (1-\hat\alpha_1 ) \hat v_D)\ \ \textrm{in}\ L^2(\Omega) \times L^2(\Omega).
\end{equation}

\noindent{\it Step 3'.}  Similar to {\it Step 3} in Section 4, we  improve the $L^2(\Omega)\times L^2(\Omega)-$convergence to $\mathbb X\times \mathbb X-$convergence here. Define
$$
\hat\eta (t) = \inf \{ \eta\ | \ u(x,t)<\eta \hat u_d(x), v(x,t) > (1-\eta ) \hat v_D(x) \ \textrm{in}\ \bar\Omega  \}.
$$
Denote
$\hat\eta ^* = \lim_{t\rightarrow \infty} \hat\eta (t),$ where $\hat\eta(t)$ is decreasing in $t$ due to comparison principle.

Since  $\{   v(\cdot, t)  \ |\  t\geq 0 \}$ is precompact in $ \mathbb X$, it follows immediately from (\ref{pf-thm13-L2}) that
\begin{equation}\label{pf-thm13-Linfty}
\lim_{t\rightarrow \infty }  v(\cdot, t)  =  (1-\hat\alpha_1 ) \hat v_D \ \ \textrm{in}\ \mathbb X .
\end{equation}
Recall that $\hat\theta_* =\hat\alpha_1<1$, hence we have the lower bound for $v(x,t)$ when $t$ is large. Then the arguments after (\ref{dispersal-lowerbd}) in the proof of Proposition \ref{prop-weakL2-nonzero} can borrowed to show that $u(x,t)<\hat u_d(x)$ in $\bar\Omega$ for large time. Therefore, w.l.o.g., we  assume that
$$
0<u_0 <\hat u_d,\ 0<v_0<\hat v_D\ \ \textrm{in}\ \bar\Omega.
$$
Then $0 <\hat\theta(0),\ \hat\eta(0)<1$

According the definitions of $\hat\theta_*$,  $\hat\eta ^*$ and $\hat\alpha_1$, it is obvious that $\hat\theta_* \leq  \hat\alpha_1 \leq \hat\eta ^*$,  $\hat\theta_*\leq 1$ and $\hat\eta^*\geq 0$.  As explained before {\it Step 1} in Section 4, it is easy to verify that
\begin{itemize}
\item {\it if $\hat\theta_* =1$,} then
$
\lim_{t\rightarrow \infty }(u(\cdot, t), v(\cdot, t)) =(  \hat u_d, 0)\ \ \textrm{in}\  \mathbb X\times \mathbb X;
$
\item {\it if $\hat\eta^* =0$},  then
$
\lim_{t\rightarrow \infty }(u(\cdot, t), v(\cdot, t)) =(  0, \hat v_D)\ \ \textrm{in}\   \mathbb X\times \mathbb X.
$
\end{itemize}
It remains to consider the case that $\hat\theta_* < 1$ and $\hat\eta^* > 0$.
To obtain the $\mathbb X\times \mathbb X-$convergence of $(u,v)$, it suffices to show $\hat\theta_* = \hat\eta ^*$.
$\hat\theta_* =\hat\alpha_1$ has been proved.

Suppose that $\hat\alpha_1 < \hat\eta ^*$. Then (\ref{pf-thm13-Linfty}) implies that there exists $T_1>0$ such that for $t\geq T_1$,
$$
v(x,t) > \left( 1- \hat\eta^* + \frac{\hat\eta^* -\hat\alpha_1}{2}\right)\hat v_D(x)\ \ \textrm{in} \ \bar\Omega.
$$
Then for $t\geq T_1$,
\begin{eqnarray*}
u_t  &=& d   \mathcal{K} [u]  +u(m(x)-b_1 u- c v)\\
   &\leq &   \mathcal{K} [u]  +u  \left(m(x)-b_1 u- c\left( 1- \hat\eta^* + \frac{\hat\eta^* -\hat\alpha_1}{2}\right)  \hat v_D(x) \right).
\end{eqnarray*}
Recall that $b_1 \hat u_d =c \hat v_D$. Then it is easy to check that $  \displaystyle \left(\hat\eta^*- \frac{\hat\eta^* -\hat\alpha_1}{2}\right) \hat u_d$ satisfies
$$
\mathcal{K} [u]  +u  \left(m(x)-b_1 u- c\left( 1- \hat\eta^* + \frac{\hat\eta^* -\hat\alpha_1}{2}\right)\hat v_D(x) \right) =0.
$$
Then it follows from Theorem \ref{thm-single} and comparison principle that there exists $T_2\geq T_1>0$ and $0< \tilde\delta<  (\hat\eta^* -\hat\alpha_1) /2 $ such that
$$
u(x, T_2) < ( \hat\eta^* -  \tilde\delta ) \hat u_d(x)\ \ \textrm{in} \ \bar\Omega.
$$
The above two inequalities contradict to the definition of $ \hat\eta^*$. Hence $\hat\alpha_1 = \hat\eta ^*$.

So far, we have proved the convergence of $(u(x,t), v(x,t))$ when $u(x,t)$ does not weakly converge to zero  in $L^2(\Omega)$.

At the end, assume that $u(x,t)$  weakly converges to zero  in $L^2(\Omega)$ as $t\rightarrow \infty$. It follows from Proposition \ref{prop-weakL2-zero} that $u(x,t)<\hat u_d(x)$ in $\bar\Omega$ for large time. W.l.o.g., assume that
$$
0<u_0 <\hat u_d,\  v_0 >0\ \ \textrm{in}\ \bar\Omega.
$$
Thus $\hat\eta (0)>0$.

Suppose that $\hat\eta^*  >0$. Again, by similar arguments in {\it Step 1} in Section 4, we obtain that   there exists a subsequence $\{s_j\}_{j\geq 1} $ with $s_j \rightarrow \infty$ as $j\rightarrow \infty$ and $\hat\alpha_1\in [0,1]$ such that
$$
\lim_{j\rightarrow \infty }(u(\cdot, s_j), v(\cdot, s_j)) =(\hat\alpha_1 \hat u_d, (1-\hat\alpha_1 ) \hat v_D)\ \ \textrm{in}\ L^2(\Omega) \times L^2(\Omega).
$$
This implies that $\hat\alpha_1 =0$ since $u(x,t)$  weakly converges to zero  in $L^2(\Omega)$ as $t\rightarrow \infty$.  Then similar arguments in {\it Step 2} in Section 4 can be applied to indicate that $\hat\eta^* = \hat\alpha_1 $, which is a contradiction.

Therefore, $\hat\eta^* =0$ and thus it follows that
$
\lim_{t\rightarrow \infty }(u(\cdot, t), v(\cdot, t)) =(  0, \hat v_D)\ \ \textrm{in}\ \mathbb X\times \mathbb X.
$
This completes the proof of Theorem \ref{thm-main-mix}.



\section{Other types of nonlocal dispersal strategies}
Theorems \ref{thm-main}, \ref{thm-main-2} and \ref{thm-main-mix}  are about environments with no flux boundary condition.
In this section, we briefly explain how to extend  these results  to nonlocal operators in hostile surroundings  or   periodic environments.

\begin{itemize}
\item {\it Hostile surroundings.}
For $\phi\in \mathbb X$, the nonlocal operator in hostile surroundings is defined as follows:

\noindent\textbf{(D)} $\ \displaystyle\mathcal{K}[\phi] = \int_{\Omega}k(x,y)\phi(y)dy-  \phi(x)$. 

\item {\it Periodic environments.} First set $\mathbb X_p = \{ \phi \in C(\mathbb R^n ) \ | \ \phi(x+ l_j e_j) = \phi(x) ,\ 1\leq j\leq n \},$ where $l_j>0$, $e_j= (e_{j1}, e_{j2}, ..., e_{jn})$ with $e_{ji}=1$ if $j=i$ and $e_{ji} =0$ if $j\neq i$. For $k(x,y): \mathbb R^n \times \mathbb R^n \rightarrow \mathbb R_+$, assume that

\noindent\textbf{(C$_p$)}  $\ k(x+ l_j e_j, y) = k(x,y- l_j e_j),\ 1\leq j\leq n $.

Now, for $\phi\in \mathbb X_p$ and $k(x,y)$ satisfies \textbf{(C1)}, \textbf{(C2)} and \textbf{(C$_p$)}, the nonlocal operator in periodic environments is defined as follows:

\noindent\textbf{(P)} $\ \displaystyle\mathcal{K}[\phi] = \int_{\mathbb R^n}k(x,y)\phi(y)dy-  \phi(x)$.

Denote $\Omega_p = [0,l_1]\times [0,l_2]\times ... \times [0,l_n]$. Then
\begin{eqnarray}\label{nonlocal-periodic}
\mathcal{K}[\phi] &=& \int_{\mathbb R^n}k(x,y)\phi(y)dy-  \phi(x)\cr
 &=& \int_{\Omega_p} \sum_{j=1}^n  \sum_{m=-\infty}^{\infty} k(x, y+ m l_j e_j) \phi(y) dy -  \phi(x).
\end{eqnarray}
\end{itemize}

Recall that when studying nonlocal operators defined in \textbf{(N)}, in fact we consider the operators defined in (\ref{K-kernel}) and (\ref{P-kernel}). Therefore, it is easy to see that Theorems \ref{thm-main} and \ref{thm-main-2} still hold for the system (\ref{original}) with nonlocal operators  in hostile surroundings  or   periodic environments.

At the end, when local dispersals are incorporated, for hostile surroundings, homogeneous Dirichlet boundary conditions should be imposed.  The proof of this case is almost the same as that  of  Theorem \ref{thm-main-mix}. The only different part is in the verification of (\ref{prop-important-general}), where Hopf boundary lemma is needed for Dirichlet boundary conditions.
Moreover, for periodic environments, it is natural to impose periodic boundary conditions when local dispersals are incorporated. Due to (\ref{nonlocal-periodic}), the proof of this case follows from that of Theorem \ref{thm-main-mix} word by word.


\end{document}